\documentclass{scrartcl}

\usepackage[utf8]{inputenc}
\usepackage[english]{babel}
\usepackage[T1]{fontenc}
\usepackage{enumerate}
\usepackage{amsmath}
\usepackage{amssymb}
\usepackage{amsthm}

\newtheorem{lemma}{Lemma}[section]
\newtheorem{theorem}[lemma]{Theorem}
\newtheorem*{theorem*}{Theorem}
\newtheorem{corollary}[lemma]{Corollary}
\newtheorem{proposition}[lemma]{Proposition}

\newtheorem{hypotheses}[lemma]{Hypotheses}
\newtheorem*{claim}{Claim}
\theoremstyle{definition}
\newtheorem{definition}[lemma]{Definition}

\theoremstyle{remark}
\newtheorem{remark}[lemma]{Remark}
\newtheorem{example}[lemma]{Example}

\DeclareMathOperator{\lip}{Lip}
\DeclareMathOperator{\dist}{dist}
\DeclareMathOperator{\diam}{diam}

\newcommand{\G}{\mathcal{G}}
\newcommand{\F}{\mathcal{F}}
\newcommand{\R}{\mathbb{R}}
\newcommand{\conv}{\operatorname{conv}}
\newcommand{\N}{\mathcal{N}}
\newcommand{\M}{\mathcal{M}}
\newcommand{\E}{\mathcal{E}}
\newcommand{\Q}{\mathcal{Q}}
\newcommand{\CAT}{CAT($\kappa$) }
\newcommand{\globlip}{\widehat{\operatorname{Lip}}}
\newcommand{\hav}{\operatorname{hav}}
\newcommand{\stern}{\operatorname{star}}

\begin{document}

\author{Christian Bargetz \and Michael Dymond \and Simeon Reich}
\title{Porosity Results for Sets of Strict Contractions on Geodesic Metric Spaces}
\maketitle
\begin{abstract}
  \noindent\textbf{Abstract.}
  We consider a large class of geodesic metric spaces, including Banach spaces, hyperbolic spaces and geodesic $\mathrm{CAT}(\kappa)$-spaces, and investigate the space of nonexpansive mappings on either a convex or a star-shaped subset in these settings. We prove that the strict contractions form a negligible subset of this space in the sense that they form a $\sigma$-porous subset. 
  For certain separable and complete metric spaces we show that a generic nonexpansive mapping has Lipschitz constant one at typical points of its domain. These results contain the case of nonexpansive self-mappings and the case of nonexpansive set-valued mappings as particular cases.\\[2mm]
\end{abstract}

\section{Introduction}
The question of existence of fixed points for nonexpansive mappings 
\[
f\colon C\to C,
\]
where $C$ denotes a certain nonempty closed subsets of a complete metric spaces $X$, has been well studied. Recall that a mapping $f$ is called \emph{nonexpansive} if it satisfies, for all $x,y\in C$, the inequality
\[
\rho(f(x),f(y))\leq \rho(x,y),
\]
where $\rho$ denotes the metric on $X$.
If $X$ is a Euclidean space and $C\subset X$ is bounded, closed and convex, Brouwer's fixed point theorem (Satz~4 in~\cite{Bro1911Abbildungen}) states that every continuous mapping $f\colon C\to C$ has a fixed point. The example
\[
T\colon C \to C, \quad Tx := (1, x_1, x_2, \ldots),
\]
where $C := \{x\in c_0\colon 0\leq x_n \leq 1\}$, shows that in infinite dimensions there are noncompact $C$ and nonexpansive mappings $f\colon C\to C$ without fixed points. 
In 1965 F.~E.~Browder showed in~\cite{Bro1965HilbertSpace} that nonexpansive mappings on the closed unit ball of the Hilbert space $\ell_2$ have a fixed point. Detailed discussions of the fixed point property for nonexpansive mappings can be found, for example, in Section~1.6 of~\cite{GR1984UniformConvexity} and in Chapter~4 of~\cite{GK1990FixedPointTheory}. More recent results are presented, for instance, in~\cite{Pia2015FixedPointProperty} and in the references cited therein.

Instead of characterizing the sets $C$ for which every nonexpansive self-mapping has a fixed point, F.~S.~De Blasi and J.~Myjak took a different approach in~\cite{DM1976Convergence, DM1989Porosity}. They raised the question of whether the typical nonexpansive mapping has a 
fixed point. To be more precise, let $C$ be a bounded, closed and convex subset of a Banach space $X$, and denote by
\[
\mathcal{M} := \left\{f\colon C\to C\colon \|f(x)-f(y)\|\leq \|x-y\| \text{ for all }x,y\in C\right\}
\]
the space of nonexpansive mappings on $C$ equipped with the metric of uniform convergence. 
In~\cite{DM1976Convergence} they proved that there is a dense $G_\delta$-set $\mathcal{M}'$ in $\mathcal{M}$ such that each $f\in\mathcal{M}'$ has a unique fixed point which is the pointwise limit of the iterates of $f$. They improved this result in~\cite{DM1989Porosity}, where they showed that there is a set $\mathcal{M}_*\subset\mathcal{M}$ with a $\sigma$-porous complement such that each $f\in\mathcal{M}_*$ has a unique fixed point which is the uniform limit of the iterates of $f$. Put in different words, these results state that a generic nonexpansive mapping $f$ on a bounded, closed and convex subset of a Banach space has a unique fixed point which is the uniform limit of the iterates of $f$.

Since Banach's fixed point theorem from 1922, see~\cite{Ban1922EnsembesAbstraits}, states that every \emph{strict contraction}, that is, a mapping
\[
f\colon C\to C\quad\text{with}\quad \rho(f(x),f(y)) \leq L \rho(x,y)\text{ and } L<1,
\]
has a unique fixed point which is the uniform limit of the iterates of $f$, the question arises whether a generic nonexpansive mapping on a bounded, closed and convex subset of a Banach space is, in fact, a strict contraction. Using the Kirszbraun-Valentine extension theorem, De Blasi and Myjak answered this natural question in the negative by showing in the aforementioned papers that if $X$ is a Hilbert space, then the set of strict contractions is $\sigma$-porous. In the recent article~\cite{BD2016:Porosity} the first two authors were able to show (by employing different methods) that this also holds true for general Banach spaces $X$.

In~\cite{Rak1962Contractive} E.~Rakotch proved a generalisation of Banach's fixed point theorem, where the Lipschitz constant can be replaced by a decreasing function. More precisely, a mapping $f\colon C\to C$ is called \emph{contractive in the sense of Rakotch} if there exists a decreasing function $\phi^f\colon [0,\diam(C)]\to [0,1]$ such that 
\[
\phi^f(t)< 1 \;\text{for}\; t>0 \quad\text{and}\quad \rho(f(x),f(y)) \leq \phi^f(\rho(x,y)) \rho(x,y)
\]
for all $x,y\in C$. Theorem~2 in~\cite{Rak1962Contractive} shows that every Rakotch contractive mapping has a unique fixed point which is the limit of the sequence of iterates of $f$. It can be shown that this fixed point is the uniform limit of the iterates of $f$.

In~\cite{RZ2001NoncontractiveMappings} the third author together with A.~J.~Zaslavski showed that there is a subset $\mathcal{M}_*\subset\mathcal{M}$ such that $\mathcal{M}\setminus\mathcal{M}_*$ is $\sigma$-porous and every $f\in\mathcal{M}_*$ is Rakotch contractive. This result can be interpreted as an explanation of the results of De Blasi and Myjak.

F.~Strobin showed in~\cite{Str2012PorousAndMeager} that in the case of an \emph{unbounded} domain $C$ this result is no longer true, but the original result of De Blasi and Myjak still holds.

In~\cite{RZ2016TwoPorosity} the Banach space $X$ has been replaced by a hyperbolic space, that is, a complete metric space together with a family of metric lines such that the resulting triangles are thin enough. In addition, in the unbounded case, a different metric on $\mathcal{M}$ is introduced and used to show that typical nonexpansive mappings are Rakotch contractive on bounded subsets.

Corresponding results, concerning the fixed point question and the prevalence of contractive mappings, for nonexpansive \emph{set-valued} mappings on star-shaped subsets of Banach and hyperbolic spaces have been presented in~\cite{BMRZ2009GenericExistence} and~\cite{PL2014ContractiveSetValued}.

The aim of the present paper is to show that in all the above cases the set of strict contractions is small in the sense that it is a $\sigma$-porous subset of the space of all nonexpansive mappings. In the case where the underlying space is separable, we further distinguish the nonexpansive mappings for which the Lipschitz constant is, in a certain sense, universally equal to one. We prove that even these mappings dominate the space of all nonexpansive mappings to the extent that they form the complement of a $\sigma$-porous subset. This 
extends~ \cite[Theorem~2.2]{BD2016:Porosity} to more general settings.

The paper is structured as follows: In Section~2 we develop the necessary background, before presenting our main results in Section~3. These statements are all obtained from a construction, given in Section~4. Finally, in Section~5 we discuss an application of our main results to set-valued nonexpansive mappings. More precisely, we prove that for several important spaces of nonexpansive set-valued mappings, the subset of strict contractions is $\sigma$-porous.

\section{Preliminaries and notations}
In this section we introduce the key concepts with which we work and establish various notations which appear throughout the paper.
\subsection{Nonexpansive mappings}
The central objects of study in this paper are spaces of nonexpansive mappings. Let $(X,\rho_{X})$ and $(Y,\rho_{Y})$ be complete metric spaces, and fix a point $\theta\in X$. By
\[
\M := \mathcal{M}(X,Y) := \{f\colon X\to Y\colon \lip(f)\leq 1\}
\]
we denote the space of nonexpansive mappings from $X$ to $Y$ equipped with the metric 
\begin{equation}\label{eq:metricDTheta}
  d_\theta (f,g) := \sup \left\{\frac{\rho_{Y}(f(x),g(x))}{1+\rho_{X}(x,\theta)} \colon x\in X\right\}.
\end{equation}
The inequalities 
\begin{align*}
  \frac{\rho_{Y}(f(x),g(x))}{1+\rho_{X}(x,\theta)} 
  & \leq \frac{\rho_{Y}(f(x),f(\theta))+\rho_{Y}(f(\theta),g(\theta))+\rho_{Y}(g(\theta),g(x))}{1+\rho_{X}(x,\theta)}\\
  & \leq \frac{\rho_{Y}(f(\theta),g(\theta))+2\rho_{X}(x,\theta)}{1+\rho_{X}(x,\theta)} 
    \leq \rho_{Y}(f(\theta),g(\theta)) + 2,
\end{align*}
which hold for all $x\in X$, show that $d_\theta$ is well defined. We note that the space $\M$ endowed with the metric $d_{\theta}$ is a complete metric space. Moreover, the topology of $\mathcal{M}$ does not depend on the particular choice of the point $\theta$: given $\theta_1\neq\theta$, the inequalities
\[
1 + \rho_{X}(x,\theta) \leq 1 + \rho_{X}(x,\theta_1) + \rho_{X}(\theta_1, \theta) 
\leq (1+\rho_{X}(x,\theta_1)) (1+\rho_{X}(\theta,\theta_1)),
\]
where $x\in X$, imply that the metrics $d_{\theta_1}$ and $d_\theta$ are Lipschitz equivalent. For a detailed discussion of the metric $d_\theta$, we refer the interested reader to~\cite{RZ2016TwoPorosity}.

\subsection{Porosity}
Our main results concern a special class of exceptional sets in metric spaces, namely the class of $\sigma$-porous sets, which were introduced in~\cite{Den1941LeconsII, dolvzenko1967granivcnye}. We define now the notion of porosity, according to~\cite{Zaj2005Porous}. In the context of a metric space, we write $B(x,r)$ for the open~ball with centre $x$ and radius $r$, and later $\overline{B}(x,r)$ for the closed~ball.
\begin{definition}
  Given a metric space $(M,d)$, a subset $A\subset M$ is called \emph{porous at a point $x\in A$} if there exist $\varepsilon_0>0$ and $\alpha>0$ such that for all $\varepsilon\in(0,\varepsilon_0)$, there exists a point $y\in B(x,\varepsilon)$ such that $B(y,\alpha\varepsilon)\cap A = \emptyset$. The set $A$ is called \emph{porous} if it is porous at all its points and $A$ is called \emph{$\sigma$-porous} if it is the countable union of porous sets.
\end{definition}
Note that this definition of porosity differs from the one in some of the aforementioned literature (e.g.,~\cite{DM1989Porosity}), where $A$ is called porous if the constants $\varepsilon_0$ and $\alpha$ are independent of the point $x$. Also, there the condition on $y$ reads as $B(y,\alpha\varepsilon)\subset (M\setminus A) \cap B(x,\varepsilon)$. This condition is equivalent to the one above as can be seen by choosing the point $y$ for a smaller $\varepsilon$ and adjusting $\alpha$ appropriately. For $\sigma$-porosity it also does not matter whether we assume that $\varepsilon_0$ and $\alpha$ are independent of the point $x$: assume we have a decomposition $A = \bigcup_{i=1}^{\infty} A_i$ and every $A_i$ is porous. For every $j,k\in\mathbb{N}$, define 
\[
A_i^{j,k} := \left\{x\in A_i\colon \varepsilon_0(x) \geq\frac{1}{j},\; \alpha(x)\geq\frac{1}{k}\right\}.
\]
Then $A= \bigcup_{i,j,k=1}^{\infty} A_{i}^{j,k}$ and each set $A_i^{j,k}$ is porous in the sense of~\cite{DM1989Porosity}. 
For a detailed discussion of the different concepts of porosity, we refer the reader to L.~Zaj\'{\i}\v{c}ek's survey article~\cite{Zaj2005Porous}. For the history of porosity, we also refer to~\cite{Bul1984DenjoyIndex, Ren1995Porosity}.

\subsection{Geodesic metric spaces}
A metric space $(X,\rho_{X})$ is called \emph{geodesic} if for every pair $x,y\in X$, there is an isometric embedding $c:[0,\rho_{X}(x,y)]\to X$ satisfying $c(0)=x$ and $c(\rho_{X}(x,y))=y$. The image of such an embedding is referred to as a \emph{metric segment} in $X$ with endpoints $x$ and $y$, and denoted by $[x,y]$. Such metric segments may not be unique and so the notation $[x,y]$ is in general not well defined. Given $\lambda\in[0,1]$ and a choice of metric segment $[x,y]$, we denote by $(1-\lambda)x\oplus\lambda y$ the unique point $z\in[x,y]$ satisfying $\rho_X(z,x)=\lambda \rho_X(x,y)$ and $\rho_X(z,y)=(1-\lambda)\rho_X(x,y)$. In places where we wish to emphasise that this point is defined according to the geodesic structure on the metric space $X$, we will write $(1-\lambda)x\oplus_{X}\lambda y$. An image of $\mathbb{R}$ by an isometric embedding is called a \emph{metric line}.

The most general setting in which the space of nonexpansive~mappings on a convex~set has so far been studied is that of a hyperbolic~space; see~\cite{RS1990Nonexpansive} and~\cite{RZ2016TwoPorosity}.
\begin{definition}
  Given a metric space $(X,\rho_{X})$ and a family $\mathcal{F}$ of metric segments in $X$, we call the triple $(X,\rho_{X},\F)$ \emph{hyperbolic} if the following conditions are satisfied:
  \begin{enumerate}[(i)]
  \item For each pair $x,y\in X$, there exists a unique metric segment $[x,y]\in\F$ joining $x$ and $y$.
  \item For all $x,y, z, w\in X$ and all $t\in[0,1]$,
    \begin{equation}\label{eq:hyperineq}
      \rho_{X}((1-t)x\oplus ty,(1-t)w\oplus tz)\leq (1-t)\rho_{X}(x,w)+t\rho_{X}(y,z).
    \end{equation}
  \item The collection $\mathcal{F}$ is closed with respect to subsegements. More precisely, for all $x,y\in X$ and $u,v\in[x,y]$ we have $[u,v]\subseteq [x,y]$.
  \end{enumerate}
\end{definition}

\begin{remark}
  \begin{enumerate}[(i)]
  \item Note that our definition of hyperbolic spaces slightly differs from the one in~\cite{RS1990Nonexpansive} since the original definition demands that every pair of points $x,y\in X$ admits a unique metric line $l\in\mathcal{F}$ such that $x,y\in l$. We note that in both variants of the definition the hyperbolic inequality~\eqref{eq:hyperineq} can be replaced with the following inequality for midpoints:
    \begin{equation*}
      \rho_{X}\left(\frac{1}{2}x\oplus\frac{1}{2}y,\frac{1}{2}x\oplus\frac{1}{2}z\right)\leq \frac{1}{2}\rho_{X}(y,z).
    \end{equation*} 
    A detailed discussion of different notions of hyperbolicity and convexity can be found in Remark~2.13 in~\cite[page~98]{Koh2005Metatheorems}.
  \item The hyperbolic inequality~\eqref{eq:hyperineq} was introduced by Busemann in~\cite{Bus1948NonpositiveCurvature} and is sometimes referred to as \emph{Busemann convexity}, cf.~\cite[p.~743]{Esp1015ContinuousSelections}.
  \end{enumerate}
\end{remark}

Nonexpansive mappings on convex and star-shaped subsets of Banach and hyperbolic spaces have been investigated in~\cite{BD2016:Porosity}, \cite{Str2014Porosity}, \cite{BMRZ2009GenericExistence} and \cite{RZ2016TwoPorosity}. 
We define below notions of convexity and star-shapedness in more general settings:
\begin{definition}\label{def:star-shapedconvex}
  Let $X$ be a metric space with metric $\rho_{X}$ and let $\F$ be a family of metric segments in $X$.
  \begin{enumerate}
  \item We say that a subset $C$ of $X$ is \emph{$\rho_{X}$-star-shaped} with respect to a point $x_{0}\in C$ if for every point $x\in C$, there is a metric segment $[x,x_{0}]\in \F$ such that $[x,x_{0}]\subseteq C$. Moreover, we write $\stern(C)$ for the set of points $y\in C$ with respect to which $C$ is $\rho_{X}$-star-shaped.
  \item We call a subset $C$ of $X$ \emph{$\rho_{X}$-convex} if for each pair $x,y\in C$ there is a metric segment $[x,y]\in \F$ such that $[x,y]\subseteq C$.
  \end{enumerate} 
\end{definition}
Clearly, convexity is stronger than star-shapedness: A set $C\subseteq X$ is $\rho_{X}$-convex if and only if $C$ is $\rho_{X}$-star-shaped with respect to $y$ for every point $y\in C$, i.e. $\stern(C)=C$. As a note of caution, we emphasise that the metric segments occurring in Definition~\ref{def:star-shapedconvex} need not be unique. Whenever we require that a metric segment $[x,y]$ be well defined, we will need to use condition \eqref{localuniqueness} of Definition~\ref{def:weaklyhyperbolic} below. Finally, let us point out that the above definitions of $\rho_{X}$-convex and $\rho_{X}$-star-shaped sets generalise the established notions in vector spaces and coincide with the notions defined for hyperbolic spaces. 

\subsection{Weakly hyperbolic spaces}
Whilst hyperbolic spaces form an important class of metric spaces, one can observe that even quite well-behaved metric spaces are excluded from this class. For an example, consider the unit sphere $\mathbb{S}^{2}$ in $\R^{3}$. For non-antipodal points $x,y\in \mathbb{S}^{2}$, there is a unique geodesic on the sphere with endpoints $x$ and $y$. However, antipodal points $x,-x\in \mathbb{S}^{2}$ admit infinitely many geodesics between them and there is no way to define the metric segment $[x,-x]$ so that the hyperbolic~inequality~\eqref{eq:hyperineq} is satisfied. Even if we relax the uniqueness condition on the family of metric segments, the sphere still presents problems. Taking $y=z$ in inequality~\eqref{eq:hyperineq}, we get
\begin{equation*}
  \rho_{X}((1-t)x\oplus ty,(1-t)w\oplus t y)\leq (1-t)\rho_{X}(x,w).
\end{equation*}
However, if we take $y\in \mathbb{S}^{2}$ to be the north pole, $x$ and $w$ to be two distinct points lying on the same line of latitude in the southern~hemisphere, we observe that
\begin{equation*}
  \rho_{\mathbb{S}^{2}}((1-t)x\oplus ty,(1-t)w\oplus ty)>\rho_{\mathbb{S}^{2}}(x,w).
\end{equation*}
for small $t\in(0,1)$. In other words, it is easy to find triangles on the sphere which become `fatter' as one moves away from their base towards their peak.

Thus, we propose to weaken the hyperbolic condition, in order to capture a larger class of metric spaces, including the sphere $\mathbb{S}^{2}$ and all geodesic $\operatorname{CAT}(\kappa)$ spaces.
\begin{definition}\label{def:weaklyhyperbolic}
  Given a metric space $(X,\rho_{X})$ and a family $\F$ of metric segments in $X$, we say that the triple $(X,\rho_{X},\F)$   is \emph{of temperate curvature} if the following conditions are satisfied:
  \begin{enumerate}[(i)]
  \item\label{localuniqueness} There exists a constant $D_{X}>0$ such that for any $x,y\in X$ with $\rho_{X}(x,y)<D_{X}$, there is at most one metric segment $[x,y]\in\F$ with endpoints $x$ and $y$. In the case where the metric segments in the family $\mathcal{F}$ are unique, we set $D_{X}=\infty$.
  \item\label{thinishtriangles} For all $x,y\in X$ with $\rho_{X}(x,y)<D_{X}$ and every $\sigma>0$, there exists a positive number $\delta_{X}=\delta_{X}(x,y,\sigma)$ such that
    \begin{equation}\label{eq:defdelta}
      \rho_{X}((1-t)z\oplus ty,(1-t)w\oplus ty)\leq (1+\sigma)\rho_{X}(z,w)
    \end{equation}
    whenever $z,w\in B(x,\delta_{X})$, $[z,y],[w,y]\in\mathcal{F}$ and $t\in [0,\delta_{X})$. 
  \end{enumerate}
  A triple $(X,\rho_{X},\mathcal{F})$ of temperate curvature is called \emph{weakly hyperbolic} if, in addition, the following conditions are satisfied:
  \begin{enumerate}[(i)]
    \setcounter{enumi}{2}
  \item\label{closedwrtsubseg} $\mathcal{F}$ is closed with respect to subsegments, that is, for all metric segments $[x,y]\in\mathcal{F}$ and all points $z,w\in[x,y]$ there is a metric segment $[z,w]\in\mathcal{F}$ with $[z,w]\subseteq [x,y]$.
  \item\label{geodesic} For all $x,y\in X$ there exists a metric segment $[x,y]\in\mathcal{F}$.
  \item\label{smallballs} For all $x\in X$ and $r\in(0,D_{X}/2)$, the ball $B(x,r)$ is a $\rho_{X}$-convex subset of $X$.
  \end{enumerate}
  When referring to either a space of temperate curvature or to a weakly hyperbolic space $(X,\rho_{X},\F)$, we often suppress the metric $\rho_{X}$ and the family of metric segments $\F$.
\end{definition}
Condition~\eqref{localuniqueness} weakens the assumption that every pair of points is connected by a unique metric segment. We note that the sphere $\mathbb{S}^{2}$ satisfies condition~\eqref{localuniqueness} with $D_{\mathbb{S}^2}=\pi$. 
Condition~\eqref{thinishtriangles} is a significant weakening of the hyperbolic inequality~\eqref{eq:hyperineq} and allows us to form `thin-ish' triangles in the space $X$. Let us imagine that we wish to form a triangle $T$ with vertices $y,z,w$ in $X$. We fix first the `peak' $y$ of the triangle $T$ and then consider an arbitrary location $x\in X$ with $\rho_{X}(x,y)<D_{X}$. Condition~\eqref{thinishtriangles} allows us to choose a small neighborhood of the point $x$ so that placing the remaining two vertices $z,w$ in this neighborhood, we produce a triangle in which the sides $[z,y]$ and $[w,y]$ do not bulge out too much as one moves a little away from the base of the triangle $[z,w]$ towards the peak~$y$.

It is clear that all hyperbolic spaces are weakly hyperbolic. We now demonstrate that the class of weakly hyperbolic spaces is significantly larger than that of hyperbolic spaces. More precisely, we  show that all geodesic \CAT spaces are weakly hyperbolic. Let us first recall the definition of \CAT spaces, from~\cite{BH1999MetricSpaces}.
\begin{definition}
  \begin{enumerate}
  \item We define a family of model spaces $(M_{\kappa})$, where $\kappa\in\R$, as follows:
    \begin{enumerate}
    \item For $\kappa>0$ we let $M_{\kappa}$ denote the metric space given by the sphere $\mathbb{S}^{2}$ with its standard path length metric, scaled by a factor of $1/\sqrt{\kappa}$.
    \item We define $M_{0}$ as the Euclidean space $\R^{2}$.
    \item For $\kappa<0$ we write $M_{\kappa}$ for the hyperbolic space $\mathbb{H}^{2}$ (see~\cite[Definition~2.10]{BH1999MetricSpaces}) with metric scaled by a factor of $1/\sqrt{-\kappa}$. 
    \end{enumerate}
    We write $d_{\kappa}$ for the metric on $M_{\kappa}$.
  \item Let $\kappa\in \R$ and $(X,\rho_{X})$ be a metric space. 
    Given three points $x_{1},x_{2},x_{3}\in X$ and metric segments of the form $[x_{1},x_{2}],[x_{2},x_{3}],[x_{3},x_{1}]\subseteq X$ we call the union $[x_{1},x_{2}]\cup[x_{2},x_{3}]\cup[x_{3},x_{1}]$ a \emph{geodesic triangle} with vertices $x_{1},x_{2},x_{3}$.
    A  geodesic triangle with vertices $\overline{x}_{1},\overline{x}_{2}, \overline{x}_{3}$ in $M_{\kappa}$ is said to be a \emph{comparison triangle} for a geodesic triangle with vertices $x_{1},x_{2},x_{3}$ in $X$ if $d_{\kappa}(\overline{x}_{i},\overline{x}_{j})=\rho_{X}(x_{i},x_{j})$. 
    A point $\overline{x}\in [\overline{x}_{i},\overline{x}_{j}]$ is called a \emph{comparison point} for $x\in[x_{i},x_{j}]$ if $d_{\kappa}(\overline{x},\overline{x}_{k})=\rho_{X}(x,x_{k})$ for $k=i,j$.
  \item Let $(X,\rho_{X})$ be a metric space. If $\kappa\leq 0$, then $(X,\rho_{X})$ is called a \CAT space if it is geodesic and every geodesic triangle $T$ in $X$ has a comparison triangle $\overline{T}$ in $M_{\kappa}$ such that
    \begin{equation}\label{eq:comparisontriangle}
      \rho_{X}(x,y)\leq d_{\kappa}(\overline{x},\overline{y})
    \end{equation}
    whenever $\overline{x},\overline{y}\in\overline{T}$ are comparison points for $x,y\in T$. If $\kappa>0$, then we define a constant $D_{\kappa}=\diam M_{\kappa}=\frac{\pi}{\sqrt{\kappa}}$ and we say that $(X,\rho_{X})$ is a \CAT space if for every pair of points $x,y\in X$ with $\rho_X(x,y)< D_\kappa$ there is a metric segment joining $x$ and $y$ and every geodesic triangle $T\subseteq X$ with perimeter smaller that $2D_\kappa$, that is, $\rho_X(x,y)+\rho_X(y,z)+\rho_X(z,x)< 2D_{\kappa}$, where $x,y,z$ denote the vertices of $T$, has a comparison triangle $\overline{T}$ in $M_\kappa$ such that~\eqref{eq:comparisontriangle} is satisfied.
  \end{enumerate}
\end{definition}
Thus, \CAT spaces can be thought of as metric spaces for which every sufficiently small geodesic triangle is `thinner' in all directions than a corresponding triangle in the model space $M_{\kappa}$. The classes of \CAT spaces are increasing in the sense that whenever $X$ is a \CAT space, it is also a $\operatorname{CAT}(\kappa')$ space for all $\kappa'\geq \kappa$; see~\cite[Theorem~1.12]{BH1999MetricSpaces}. 

In the proof of the next proposition, the most difficult task is to establish that every \CAT space  satisfies condition~(\ref{thinishtriangles}) of Definition~\ref{def:weaklyhyperbolic} and, in particular, to verify inequality~\eqref{eq:defdelta}. A related inequality for geodesic triangles with side lengths smaller than $\pi/2$ in CAT($1$) spaces is shown in Lemma~3.3 of~\cite{Pia2011HalpernIteration}.

\begin{proposition}\label{prop:CATisweaklyhyperbolic}
  Every geodesic \CAT space is weakly hyperbolic.
\end{proposition}
\begin{proof}
  Let $(X,\rho_{X})$ be a \CAT space and $\mathcal{F}$ be the collection of all geodesics in $X$. We show that the triple $(X,\rho_{X},\mathcal{F})$ is a weakly hyperbolic space.  We may assume that $\kappa>0$. It is already clear that the family $\mathcal{F}$ satisfies conditions~\eqref{closedwrtsubseg} and \eqref{geodesic} of Defintion~\ref{def:weaklyhyperbolic}. For a proof that $X$ satisfies 
  conditions~\eqref{localuniqueness} and \eqref{smallballs} with $D_{X}=D_{\kappa}$ we refer the reader to \cite[Proposition~1.4]{BH1999MetricSpaces}. 
  We now verify that $X$ satisfies condition~\eqref{thinishtriangles}. As a first step, we show that it is sufficient to verify that the sphere $\mathbb{S}^{2}$ with metric $\rho=\rho_{\mathbb{S}^{2}}$ scaled by $1/\sqrt{\kappa}$ satisfies this condition. Suppose that the model spaces satisfy conditon~(\ref{thinishtriangles}) of Definition~\ref{def:weaklyhyperbolic}. Let $(X,\rho_{X})$ be a \CAT space and let $x,y\in X$ with $\rho_{X}(x,y)<D_{\kappa}$. Then we choose $\overline{x},\overline{y}\in M_{\kappa}$ with $d_{\kappa}(\overline{x},\overline{y})=\rho_{X}(x,y)$. Given $\sigma>0$, we choose $\delta=\delta_{X}(x,y,\sigma)\in(0,\delta_{\kappa}(\overline{x},\overline{y},\sigma)/4)$, where $\delta_{\kappa}(\overline{x},\overline{y},\sigma)$ is given by condition (ii) for $M_{\kappa}$, sufficiently small so that $\rho_{X}(x,y)+2\delta<D_{\kappa}$. Let $z,w\in B(x,\delta)$. Then by the triangle inequality, we have
  \[
    \rho_X(y,w)+\rho_X(w,z)+\rho_X(z,y) < 2(\rho_X(x,y)+2\delta) < 2D_\kappa.
  \]
  Therefore we can choose a comparison triangle in $M_{\kappa}$ with vertices $\overline{y}',\overline{z},\overline{w}$ for the geodesic triangle with vertices $y,z,w$ in $X$.
  Since $d_{k}(\overline{u},\overline{y}')=\rho_{X}(u,y)$ for $u\in\left\{z,w\right\}$ and $z,w\in B(x,\delta)$, we have
  \begin{equation*}
    \left|d_{\kappa}(\overline{u},\overline{y}')-\rho_{X}(x,y)\right|< \delta
  \end{equation*}
  for $u\in\left\{z,w\right\}$. It follows that there is a great circle passing through $\overline{z}$ and $\overline{y}'$ and a point $\overline{x}'$ on this great circle with \begin{equation*}\label{eq:condxp}
    d_{\kappa}(\overline{x}',\overline{y}')=\rho_{X}(x,y)\quad\text{and}\quad d_{\kappa}(\overline{x}',\overline{z})<\delta.
  \end{equation*}
  Since the metric $d_{\kappa}$ on $M_{\kappa}$ is invariant under isometries of the sphere, we may assume now that $\overline{y}'=\overline{y}$ and $\overline{x}'=\overline{x}$. Then we have $\overline{z},\overline{w}\in B(\overline{x},4\delta)\subset B(\overline{x},\delta_{\kappa}(\overline{x},\overline{y},\sigma))$. Therefore, by condition (ii) for $M_{\kappa}$, we get
  \begin{equation*}
    d_{\kappa}((1-t)\overline{z}\oplus t\overline{y},(1-t)\overline{w}\oplus t\overline{y})\leq (1+\sigma)d_{k}(\overline{z},\overline{w})=(1+\sigma)\rho_{X}(z,w)
  \end{equation*}
  for all $t\in[0,\delta_{\kappa}(\overline{x},\overline{y},\sigma))$. Consequently, by~\eqref{eq:comparisontriangle},
  \begin{equation*}
    \rho_{X}((1-t)z\oplus ty,(1-t)w\oplus ty)\leq d_{\kappa}((1-t)\overline{z}\oplus t\overline{y},(1-t)\overline{w}\oplus t\overline{y})\leq (1+\sigma)\rho_{X}(z,w)
  \end{equation*}
  for all $t\in[0,\delta)\subseteq(0,\delta_{\kappa}(\overline{x},\overline{y},\sigma))$. 

  From this point on we will assume that $\kappa=1$, since multiplying the metric $\rho$ on the sphere by a factor of $1/\sqrt{\kappa}$ does not affect any of the calculations which follow. 

  Let $x,y\in \mathbb{S}^{2}$ with $\rho(x,y)<D_{\kappa}=\pi$ and fix $\sigma>0$. Note that $x$ and $y$ cannot be antipodal. We consider two cases, namely $\rho(x,y)>0$ and $\rho(x,y)=0$. We start with the case $\rho(x,y)>0$ and let $\delta=\delta(x,y,\sigma)\in(0,\pi/8)$ be some positive constant to be determined later in the proof. For now we just prescribe that $\delta$ be small enough so that
  \begin{equation*}
    I(x,y,\delta):=[(1-\delta)(\rho(x,y)-\delta),\rho(x,y)+\delta]\subseteq (0,\pi).
  \end{equation*}
  We define constants $m_{\sin}=m_{\sin}(x,y,\delta)$ and $M_{\sin}=M_{\sin}(x,y,\delta)$ by
  \begin{align*}
    m_{\sin}&:=\min\left\{\sin \theta \colon \theta\in I(x,y,\delta)\right\},\\ 
    M_{\sin}&:=\max\left\{\sin \theta \colon \theta\in I(x,y,\delta)\right\},
  \end{align*}
  and define constants $m_{\cos}$, $M_{\cos}$ analogously with $\sin$ replaced by $\cos$. Note that 
  \[
  m_{\sin},M_{\sin}\to\sin\rho(x,y)\quad\text{and}\quad m_{\cos},M_{\cos}\to\cos\rho(x,y)
  \]
  as $\delta\to 0^{+}$. 

  For points $z\in B(x,\delta)$ we write $|z|=\rho(z,y)$. We note that $\left||z|-|x|\right|\leq \rho(z,x)\leq \delta$ and hence $|z|\in I(x,y,\delta)\subseteq(0,\pi)$ for all $z\in B(x,\delta)$. For points $z,w\in B(x,\delta)$, we let $\Theta(z,w)$ denote the angle at the vertex $y$ of the spherical triangle with vertices $z,w,y$. In what follows we use the spherical law of cosines~\cite[Proposition~2.4.1]{Jen1994ModernGeometry} and the equivalent law of haversines:
  \begin{equation*}
    \cos c=\cos a\cos b+\sin a\sin b\cos C,\quad \hav c=\hav(a-b)+\sin a\sin b\hav C,
  \end{equation*}
  where $\hav\theta:=\sin^{2}(\theta/2)$, which relate the side lengths $a$, $b$, $c$ of a spherical triangle to the angle $C$ at the vertex opposite to the side of length $c$. 

  Let $z,w\in B(x,\delta)$. Applying the spherical law of cosines to the spherical triangle with vertices $z$, $w$ and $y$, we deduce that
  \begin{equation*}
    1\geq\cos \Theta(z,w)=\frac{\cos \rho(z,w)-\cos |z|\cos |w|}{\sin |z|\sin |w|}
    \geq \frac{\cos2\delta-\max\left\{M_{\cos}^{2},m_{\cos}^{2}\right\}}{M_{\sin}^{2}}
  \end{equation*}
  provided we choose $\delta$ small enough so that $\cos2\delta-\max\left\{M_{\cos}^{2},m_{\cos}^{2}\right\}\geq 0$.
  In the above we use the facts that $\cos $ is decreasing on the interval $(0,\pi/2)$ and $\rho(z,w)\leq 2\delta$. The last expression is independent of $z,w\in B(x,\delta)$ and converges to 1. 
  It follows that
  \begin{equation}\label{eq:anglecontrol}
    \sup\left\{\Theta(z,w) \colon z,w\in B(x,\delta)\right\}\to0\quad\text{ as }\quad\delta\to {0^{+}}.
  \end{equation}
  For $t\in [1-\delta,1]$ and $z,w\in B(x,\delta)$, we consider the spherical triangle with vertices $z_{t}:=tz\oplus (1-t)y$, $w_{t}:=tw\oplus (1-t)y$, and $y$. This triangle has sides of length $t\left|z\right|$, $t\left|w\right|$ and $\rho(z_{t},w_{t})$, and angle $\Theta(z,w)$ at the vertex $y$. Without loss of generality, we assume $|z|\geq |w|$ and note that the inequalities $||u|-|x||\leq \delta$ for all $u\in B(x,\delta)$ and $1-\delta\leq t \leq 1$ together with the definition of $I(x,y,\delta)$ imply that $t|z|,t|w|\in I(x,y,\delta)$. In addition note that $||z|-|w||\leq 2\delta<\pi/4$ by the triangle inequality and hence $|z|-|w|\in [0,\pi/4)$. Using the law of haversines, we obtain
  \begin{align*}
    &\frac{\hav \rho(z_{t},w_{t})}{\hav \rho(z,w)}=\frac{\hav(t(|z|-|w|))+\sin t|z|\sin t|w| 
      \hav \Theta(z,w)}{\hav (|z|-|w|)+\sin |z|\sin |w|\hav \Theta(z,w)}\\
    &\leq 1+\frac{(\hav(t(|z|-|w|))-\hav(|z|-|w|))+(\sin t|z|\sin t|w|-\sin |z|\sin |w|) 
      \hav \Theta(z,w)}{\hav (|z|-|w|)+\sin |z|\sin |w|\hav \Theta(z,w)}\\
    &\leq 1+\frac{M_{\sin}^{2}-m_{\sin}^{2}}{m_{\sin}^{2}}.
  \end{align*}
  To deduce the above inequalities we use the fact that $\hav$ is monotonically increasing and non-negative on the interval $[0,\pi/2)$ in combination with the constraints on $\delta$, $|z|$, $|w|$ and $t$ as discussed above. Observe that the last expression converges to $1$ as $\delta\to {0^{+}}$ and is independent of the choices of $z,w\in B(x,\delta)$ and $t\in[1-\delta,1]$. 
  Given $\eta>0$ to be determined later in the proof, it follows that we can choose $\delta$ sufficiently small depending only on the points $x,y$ so that 
  \begin{equation}\label{eq:havct/havc}
    \frac{\hav \rho(z_{t},w_{t})}{\hav \rho(z,w)}\leq 1+\eta \qquad \forall t\in[1-\delta,1],\quad \forall z,w\in B(x,\delta).
  \end{equation}
  Next, observe that
  \begin{align*}
    \hav \rho(z_{t},w_{t})&=\hav(t(|z|-|w|))+\sin t|z|\sin t|w|\hav \Theta(z,w)\\
                          &\leq \hav(2\delta)+M_{\sin}^{2}\hav(\sup\left\{\Theta(z,w)\colon z,w\in B(x,\delta)\right\})
  \end{align*}
  since $t(|z|-|w|)<2\delta<\pi/2$ and $\hav$ is increasing on $[0,\pi/2)$.
  The last expression is independent of the choices of $z,w\in B(x,\delta)$ and $t\in[1-\delta,1]$ and, using~\eqref{eq:anglecontrol}, we see that it converges to $0$ as $\delta\to {0^{+}}$. Thus, using that $\hav \rho(z_t,w_t)\to 0$ implies $\rho(z_t,w_t)\to 0$ and the Taylor expansion of $\hav x $ at $x=0$, we can prescribe that $\delta>0$ be sufficiently small so that the following inequalities hold:
  \begin{equation}\label{eq:havctapp}
    \hav \rho(z_{t},w_{t})\geq \frac{\rho(z_{t},w_{t})^{2}}{4}-\eta \rho(z_{t},w_{t})^{2}\qquad \forall z,w\in B(x,\delta),
    \quad \forall t\in[1-\delta,1],
  \end{equation}
  \begin{equation}\label{eq:havcapp}
    \hav \rho(z,w)\leq \frac{\rho(z,w)^{2}}{4}+\eta \rho(z,w)^{2} \qquad \forall z,w\in B(x,\delta).
  \end{equation}
  Combining inequalities \eqref{eq:havct/havc}, \eqref{eq:havctapp} and \eqref{eq:havcapp}, we deduce that
  \begin{equation*}
    \rho(z_{t},w_{t})^{2}\leq \frac{(1+\eta)(\frac{1}{4}+\eta)}{(\frac{1}{4}-\eta)}\rho(z,w)^{2} \qquad 
    \forall z,w\in B(x,\delta),\quad\forall t\in[1-\delta,1].
  \end{equation*}
  If we prescribe that $\eta$ be chosen sufficiently small so that the constant before $\rho(z,w)^{2}$ in the above inequality is at most $(1+\sigma)^{2}$, then we obtain the desired result.

  If $\rho(x,y)=0$, we choose $\delta=\delta_{X}(x,y,\sigma)\in(0,\pi/4)$. Given $z,w\in B(x,\delta)$ and $t\in(0,1)$ we let $z_{t}:=tz\oplus (1-t)x$, $w_{t}:=tw\oplus (1-t) x$ and $\theta$ be the angle at the vertex $x$ of the spherical triangle with vertices $x,w,z$. For $u\in \mathbb{S}^{2}$ we also write $\left|u\right|$ for the distance $\rho(u,x)$. Then, for all $t\in[0,1]$, the law of haversines gives
  \begin{align*}    \hav\rho(z_{t},w_{t})&=\hav(t(\left|z\right|-\left|w\right|))+\sin(t\left|z\right|)\sin(t\left|w\right|)\hav \theta\\
                      &\leq\hav(\left|z\right|-\left|w\right|)+\sin\left|z\right|\sin\left|w\right|\hav\theta\\
                      &=\hav\rho(z,w).
  \end{align*}
  In the above we used that $\hav$ is symmetric, non-negative and that $\hav$ and $\sin$ are increasing on the interval $[0,\pi/2)$. Using again that $\hav$ is increasing on the interval $[0,\pi/2)$, we conclude that $\rho(z_{t},w_{t})\leq\rho(z,w)$ for all $t\in[0,1]$. This is a stronger version of the inequality in Definition~\ref{def:weaklyhyperbolic},~(\ref{thinishtriangles}).  
\end{proof}

Given a subset $E\subset X$ of a metric space $X$ and $r>0$, we use the notations
\[
  B(E,r) := \{x\in X\colon d(x,E) <r\}\qquad\text{and}\qquad \overline{B}(E,r) := \{x\in X\colon d(x,E) \leq r\}.
\]
Note that if $X$ is a weakly hyperbolic space and $E\subseteq X$ is a nonempty subset, the set $\overline{B}(E,r)\setminus B(E,r)$ has empty interior. Indeed, any $x\in\overline{B}(E,r)\setminus B(E,r)$ satisfies 
\begin{equation*}
  \dist(x,E):=\inf\left\{\rho_{X}(x,u)\colon u\in E\right\}=r.
\end{equation*}
Given $0<\varepsilon<r$ we choose $x_{0}\in E$ such that $r\leq\rho_{X}(x,x_{0})<r+\varepsilon/2$. Then every point of the form  $(1-\frac{\varepsilon}{\rho_{X}(x,x_{0})})x\oplus\frac{\varepsilon}{\rho_{X}(x,x_{0})}x_{0{}}$ lies in $\overline{B}(x,\varepsilon)\cap B(E,r)$. This shows that $\overline{B}(E,r)\setminus B(E,r)$ has empty interior. Note that the above argument also shows that for a $\rho_X$-star-shaped set $C\subset X$ and any $r>0$, $\overline{B}(\stern(C),r)\setminus B(\stern(C),r)$ has empty interior in $C$.

In addition, we get that in weakly hyperbolic spaces the closure of an open ball is the corresponding closed ball, that is, we have $\overline{B(x,r)} = \overline{B}(x,r)$ for all $x\in X$ and all $r>0$. The inclusion $\overline{B(x,r)}\subseteq \overline{B}(x,r)$ follows from the continuity of the metric whereas the opposite inclusion can be deduced analogously to the above argument using the fact that $[z,x]\setminus\left\{z\right\}\subseteq B(x,r)$ for any $z\in \overline{B}(x,r)$.

\subsection{$\ell^{\infty}$ spaces}
We make frequent use of two special properties of $\ell_{\infty}$ spaces. Firstly, we exploit the fact that any metric space can be isometrically embedded into $\ell_{\infty}(\Omega)$ for some set $\Omega$. Thus, we often identify metric spaces with subsets of some $\ell_{\infty}$ space. Note that given two metric spaces $X$ and $Y$ which are isometrically embedded into $\ell_{\infty}(\Omega_1)$ and $\ell_{\infty}(\Omega_2)$, respectively, we can embed both $X$ and $Y$ isometrically into $\ell_{\infty}(\Omega_1\uplus\Omega_2)$, where $\Omega_1\uplus\Omega_2$ stands for the disjoint union of $\Omega_1$ and $\Omega_2$ since $\ell_{\infty}(\Omega_i)$, $i=1,2$, embeds isometrically into $\ell_{\infty}(\Omega_1\uplus\Omega_2)$. Secondly, we make use of the fact that any Lipschitz mapping defined on a subset of a metric space $M$ and taking values in some $\ell_{\infty}(\Omega)$, can be extended to a Lipschitz mapping $F\colon M\to\ell_{\infty}(\Omega)$ with the same Lipschitz constant. A detailed discussion of these special properties of $\ell_{\infty}$ spaces can be found in~\cite[Chapter 1]{BL2000GeomtericNonlinear}.

\section{Main results}
In this section we present our main results. In fact we show that all of our main results can be derived 
from a single theorem, Theorem~\ref{thm:all}, which is proved in the next section. Before stating this 
result, we establish our general hypotheses.
\begin{hypotheses}\label{hypotheses}
  Let $(X,\rho_{X})$ be a complete, weakly hyperbolic space, $(Y,\rho_{Y})$ be a complete space of temperate curvature
  and $C_{X}\subseteq X$, $C_{Y}\subseteq Y$ be non-empty, closed, non-singleton and $\rho_{X}$- and $\rho_{Y}$-star-shaped 
  subsets of $X$ and $Y$, respectively. 
  Suppose that the set $C_{Y}$ satisfies $C_{Y}\subseteq B(\stern(C_{Y}),D_{Y})$. Let $\conv(C_{X})$ denote a 
  $\rho_{X}$-convex subset of $X$ containing $C_{X}$ and choose a set $\Omega$ so that $X,Y\subset \ell_{\infty}(\Omega)$. 
  Let $\theta\in X$ and $\M(C_{X},C_{Y})$ denote the space of nonexpansive mappings from $C_{X}$ to $C_{Y}$, equipped with 
  the metric~$d_{\theta}$. Let $\N(C_{X},C_{Y})$ denote the subset of $\M(C_{X},C_{Y})$ formed by the strict contractions.
  Given a mapping $f\in \M(C_{X},C_{Y})$, we let $\E(f)$ denote the set of all $1$-Lipschitz extensions 
  $F\colon\conv(C_{X})\to \ell_{\infty}(\Omega)$ of~$f$. 	
\end{hypotheses}
We note that the condition $C_{Y}\subseteq B(\stern(C_{Y}),D_{Y})$ is satisfied in particular in each of the following cases:
\begin{itemize}
\item $C_{Y}$ is $\rho_{Y}$-convex,
\item $Y$ is a space of temperate curvature with $D_{Y}=\infty$. This class of spaces includes all hyperbolic spaces and 
  \CAT spaces with $\kappa\leq 0$. 
\end{itemize}
In what follows, given a set $U$ and a Lipschitz mapping $f$, we write $f|_{U}$ for the restriction of $f$ to the subset of its domain contained in $U$.
\begin{theorem}\label{thm:all}
  Let $U$ be an open subset of $X$ with $U\cap C_{X}\neq\emptyset$ and $U\subseteq B(\stern(C_{X}),D_{X})$. Then the set
  \begin{equation*}
    \Q(U)=\left\{f\in \M(C_{X},C_{Y})\colon \inf_{F\in\E(f)}\lip(F|_{U})<1\right\}
  \end{equation*}
  is $\sigma$-porous in $\M(C_{X},C_{Y})$.
\end{theorem}
As a corollary of the above theorem, we obtain the $\sigma$-porosity of the set $\N(C_{X},C_{Y})$ in the space 
$\M(C_{X},C_{Y})$:
\begin{theorem}\label{thm:Nsigporous}
  The set $\N(C_{X},C_{Y})$ is a $\sigma$-porous subset of $\M(C_{X},C_{Y})$. 
\end{theorem}
\begin{proof}
  Any strict contraction $f\colon C_{X}\to C_{Y}$ can be extended to a strict contraction 
  $F\colon\conv(C_{X})\to\ell_{\infty}(\Omega)$. Therefore $\N(C_{X},C_{Y})\subseteq \Q(U)$, where $U$ may be
  chosen arbitrarily satisfying the conditions of Theorem~\ref{thm:all}.
\end{proof}
Whilst Theorem~\ref{thm:Nsigporous} tells us that nearly all mappings in $\M(C_{X},C_{Y})$ have the maximal 
permitted Lipschitz~constant~one, we note that a large Lipschitz constant can be achieved through sporadic 
behavior. It is easy to find examples of mappings with a large Lipschitz constant that, when restricted to 
a large subset of their domain, behave like strict contractions or even constant mappings. Thus, we now 
consider the question of the size of the set of mappings in $\M(C_{X},C_{Y})$ for which a large set of points 
in $C_{X}$ in some sense witnesses the maximal Lipschitz constant. The paper~\cite{BD2016:Porosity} proves 
that, for a non-empty, non-singleton, closed, convex and bounded subset~$C$ of a separable~Banach~space~$X$, 
there is a $\sigma$-porous subset of the space $\M(C,C)$, outside of which all mappings $f$ admit a residual 
subset of $C$ on which the quantity
\begin{equation*}
  \lip(f,x):=\limsup_{r\to 0^{+}}\left\{\frac{\rho_{Y}(f(y),f(x))}{\rho_{X}(x,y)}\colon y\in B(x,r)\setminus\left\{x\right\}\right\}
\end{equation*} 
is uniformly one. We use the term \emph{residual} here in the sense of the Baire Category Theorem. The proof 
of this result makes essential use of the fact that the Lipschitz constant of a mapping on a convex set $C$ 
can be expressed as the supremum of $\lip(f,x)$ over all points $x\in C$. We verify this property for Lipschitz 
mappings on convex subsets of $X$:
\begin{lemma}\label{lem:LipOnGamma}
  Let $C$ be a non-empty, non-singleton, $\rho_{X}$-convex subset of $X$. Given a Lipschitz mapping $f:C\to Y$ and 
  a number $0<L<\lip(f)$, 
  there exist points $u_{0},u_{1}\in C$ such that
  \begin{equation}\label{eq:liminfGamma}
    \liminf_{t\to 0^+}\frac{\rho_{Y}(f((1-t)u_0\oplus tu_1), f(u_0))}{t\rho_{X}(u_0,u_1)} > L.
  \end{equation}
  In the case where $C\subseteq [w_{0},x_{0}]$ for some $w_{0},x_{0}\in X$, then such points $u_{0},u_{1}\in C$ can be 
  found with $u_{1}=x_{0}$.
\end{lemma}
\begin{proof}
  Let $L'\in(L,\lip(f))$ and choose points $v,w\in C$ such that
  \begin{equation*}
    \frac{\rho_{Y}(f(w),f(v))}{\rho_{X}(v,w)}>L'.
  \end{equation*}
  In the case where $C\subseteq[w_{0},x_{0}]$, we identify the metric segment $[w_{0},x_{0}]$ with a closed interval in $\R$ and additionally prescribe that $v<w<x_{0}$. 
  
  Let $[v,w]$ be a metric segment in $X$ with endpoints $v$ and $w$. We identify $[v,w]$ 
  with a closed interval in $\mathbb{R}$. Assume that
  \begin{equation}\label{eq:LipConstContradictionGamma}
    \liminf_{t\to 0^+} \frac{\rho_{Y}(f((1-t)u_0\oplus t w),f(u_0))}{t\rho_{X}(u_0,w)} < L'
  \end{equation}
  for all $u_0\in [v,w)$, where $[u_0,w]\subseteq[v,w]$. We define a collection of metric segments $\mathcal{U}$ by
  \[
  \mathcal{U} := \left\{[\xi,\eta]\subset (v,w) \colon \frac{\rho_{Y}(f(\xi),f(\eta))}{\rho_{X}(\eta,\xi)} <L'\right\},
  \]
  which is, by assumption~\eqref{eq:LipConstContradictionGamma}, a Vitali cover of $(v,w)$. By Vitali's
  covering theorem, there exist pairwise disjoint intervals $[\xi_i,\eta_i]\in\mathcal{U}$ such that
  \[
  \lambda^1 \Big( (v,w)\setminus\bigcup_{i=1}^{\infty}[\xi_i,\eta_i]\Big) = 0,
  \]
  where $\lambda^1$ denotes the one-dimensional Lebesgue measure. We will prove that 
  \begin{equation*}
    \frac{\rho_{Y}(f(w),f(v))}{\rho_{X}(v,w)} \leq L',
  \end{equation*}
  contradicting the choice of $v,w\in C$. From this contadiction we conclude that  assumption~\eqref{eq:LipConstContradictionGamma}
  is false. Consequently, there exists $u_{0}\in[v,w)$ such that \eqref{eq:liminfGamma} is satisfied with $u_{1}=w$. 
  In the case $C\subseteq [w_{0},x_{0}]$ we have $u_{0}<w<x_{0}$ and therefore \eqref{eq:liminfGamma} is also satisfied 
  with $u_{1}=x_{0}$.
  
  To complete the proof, we establish the contradiction described above. For $\varepsilon>0$, choose $N$ large enough 
  so that
  \[
  \lambda^1\Big((v,w)\setminus \bigcup_{i=1}^{N} [\xi_i,\eta_i]\Big) <\frac{\varepsilon \rho(v,w)}{\lip(f)}.
  \]
  Without loss of generality, we may assume that $\xi_1<\eta_1<\xi_2<\dots<\xi_N<\eta_N$, that is, the above 
  intervals are in ascending order. From this, we deduce that
  \begin{align*}
    \frac{\rho_{Y}(f(w),f(v))}{\rho_{X}(v,w)} 
    & \leq \frac{1}{\rho_{X}(v,w)}\Big( \rho_{Y}(f(v),f(\xi_1)) +\sum_{i=1}^{N} \rho_{Y}(f(\xi_i),f(\eta_i))\\ 
    &\hspace{4.5cm} +\sum_{i=1}^{N-1}\rho_{Y}(f(\eta_i),f(\xi_{i+1}))+\rho_{Y}(f(\eta_N),f(w))\Big) \\
    & \hspace{-2.7cm} \leq \frac{1}{\rho_{X}(v,w)} \Big(L' \Big(\sum_{i=1}^{N}\rho_{X}(\xi_i,\eta_i)\Big)
      +\lip(f)\Big(\rho_{X}(v,\xi_1)+\sum_{i=1}^{N-1}\rho_{X}(\eta_i,\xi_{i+1})+\rho_{X}(\eta_N,w)\Big)\Big)\\
    & \hspace{-2.7cm} \leq \frac{1}{\rho_{X}(v,w)} (L'\rho_{X}(v,w)+\varepsilon \rho_{X}(v,w)) = L'+\varepsilon.
  \end{align*}
  Letting $\varepsilon\to 0^{+}$, we arrive at the desired contradiction. 
\end{proof}

For $\rho_{X}$-star-shaped domains, the conclusion of Lemma~\ref{lem:LipOnGamma} is, in general, not valid and so the 
global Lipschitz constant may not be approximated by $\lip(f,x)$. We demonstrate this with an example:
\begin{example} 
Let $e=(1,0)\in\mathbb{R}^{2}$ and $u\in\mathbb{S}^{1}$ with $\|e-u\|=\frac{1}{3}$. We set $A=[0,e]$, 
$B=[0,u]$, $X=A\cup B$ and define
\[
f\colon X\to X,\quad z=(z_1,z_2)\mapsto \begin{cases}(0,0)&\text{for } z\in B\\ 
  \frac{1}{2}\left(\max\left\{z_1-\frac{1}{3},0\right\},0\right) 
  & \text{for } z\in A\end{cases}.
\]
Then $\lip(f,x)$ is bounded above by $\frac{1}{2}$ for all $x\in X$ but $\|f(e)-f(u)\| = \frac{1}{3} = \|e-u\|$ 
shows that the global Lipschitz constant of $f$ is at least $1$.
\end{example}

Thus, for $\rho_{X}$-star-shaped domains, we consider a weaker control of the Lipschitz constant at a point. Namely, 
for $f\in\M(C_{X},C_{Y})$ and $x\in C_{X}$, we define the quantity 
\begin{equation*}
  \globlip(f,x):=\sup\left\{\frac{\rho_{Y}(f(y),f(x))}{\rho_{X}(x,y)}\colon y\in C_{X}\setminus\left\{x\right\}\right\},
\end{equation*}
which satisfies $\lip(f,x)\leq\globlip(f,x)$. Given a mapping $f\in\M(C_{X},C_{Y})$, we define sets 
$R(f),\widehat{R}(f)\subseteq C_X$ by
\begin{equation*}
  R(f):=\left\{x\in C_{X}\colon \lip(f,x)=1\right\},\qquad \widehat{R}(f):=\left\{x\in C_{X}\colon \globlip(f,x)=1\right\}.
\end{equation*}
We note that $R(f)\subseteq \widehat{R}(f)$. Under suitable additional assumptions we show that for nearly all 
mappings $f\in\M(C_{X},C_{Y})$, either the set $R(f)$ or the set $\widehat{R}(f)$ is a residual subset of $C_{X}$. For 
a given $f\in\M(C_{X},C_{Y})$, we point out that the sets $R(f)$ and $\widehat{R}(f)$ are both $G_{\delta}$ subsets 
of~$C_{X}$. 
To see this, note that
\begin{equation}\label{eq:Rfbarint}
  \widehat{R}(f)=\bigcap_{q\in\mathbb{Q}\cap(0,1)}\left\{x\in C_{X}\colon \globlip(f,x)>q\right\}
\end{equation}
and 
\begin{equation}\label{eq:Rfint}
  R(f)=\bigcap_{q,r\in\mathbb{Q}\cap(0,1)}\left\{x\in C_{X}\colon \lip(f,x,r)>q\right\},
\end{equation}
where for $x\in C_{X}$ and $r>0$, we define
\begin{equation*}
  \lip(f,x,r):=\sup\left\{\frac{\rho_{Y}(f(y),f(x))}{\rho_{X}(x,y)}\colon y\in C_{X}\cap B(x,r)\setminus\left\{x\right\}\right\}.
\end{equation*}
Note that we have $\lip(f,x)=\lim_{r\to0^{+}}\lip(f,x,r)$.
It is readily verified that each of the sets participating in the above intersections is open in $C_{X}$. 

In the case where the set $C_{X}$ is separable and $\rho_{X}$-convex we obtain the following generalisation of~\cite[Theorem~2.2]{BD2016:Porosity}: 
\begin{theorem}\label{thm:convexresidual}
  Suppose $C_{X}$ is separable and $\rho_{X}$-convex. Then there exists a $\sigma$-porous set $\widetilde{\N}\subseteq \M(C_{X},C_{Y})$ such that for every $f\in\M(C_{X},C_{Y})\setminus \widetilde{\N}$, the set
  \begin{equation*}
    R(f)=\left\{x\in C_{X}\colon \lip(f,x)=1\right\}
  \end{equation*} 
  is a residual subset of $C_{X}$.
\end{theorem}
\begin{proof}
  For each open set $U\subseteq X$ of diameter smaller than $D_{X}$ and non-empty intersection with $C_{X}$, we apply Theorem~\ref{thm:all} with $\conv(C_{X})=C_{X}$. Note that $\conv(C_X)=C_{X}$ implies in particular that $U\subset B(\stern(C_X),D_X)$ holds. With these settings we have $\E(f)=\left\{f\right\}$ for all $f\in\M(C_{X},C_{Y})$ and Theorem~\ref{thm:all} asserts that the set
  \begin{equation*}
    \Q(U)=\left\{f\in\M(C_{X},C_{Y})\colon \lip(f|_{U})<1\right\}
  \end{equation*}
  is a $\sigma$-porous subset of $\M(C_{X},C_{Y})$.
  
  Fix a countable dense subset $\Delta$ of $C_{X}$ and define the set $\widetilde{\N}$ by
  \begin{equation*}
    \widetilde{\N}:=\bigcup_{i=1}^{\infty}\Q(U_{i}),
  \end{equation*}
  where $(U_{i})_{i=1}^{\infty}$ is an enumeration of all sets of the form $B(x,r)$ where $x\in \Delta$ and $r\in\mathbb{Q}\cap (0,D_{X}/2)$. It is clear that $\widetilde{\N}$ is a $\sigma$-porous subset of $\M(C_{X},C_{Y})$. 
  
  Let $f\in\mathcal{M}(C_{X},C_{Y})\setminus \widetilde{\N}$. To complete the proof, we need to verify 
  that the set $R(f)$ is a residual subset of $C_{X}$. It suffices to show that each of the open subsets of $C_{X}$ occurring in the intersection in \eqref{eq:Rfint} is a dense subset of $C_{X}$. To this end, fix an open subset $U$ of $X$ such that $U\cap C_{X}\neq\emptyset$. Given $q,r\in\mathbb{Q}\cap(0,1)$, we need to show that the set
  \begin{equation*}
    T_{q,r}:=\left\{x\in C_{X}\colon \lip(f,x,r)>q\right\}
  \end{equation*}
  has non-empty intersection with $U$. Choose $j\geq 1$ so that $U_{j}\subset U$. Since $f\notin\Q(U_{j})$, we have $\lip(f|_{U_{j}})= 1$. Using the condition~\eqref{smallballs} of Definition~\ref{def:weaklyhyperbolic} on the weakly hyperbolic space $X$, we see that $C_{X}\cap U_{j}$ is $\rho_{X}$-convex, as an intersection of two $\rho_{X}$-convex sets. Therefore, we may apply Lemma~\ref{lem:LipOnGamma} with $C=C_{X}\cap U_{j}$ and deduce that there exists a point $u_{0}\in C_{X}\cap U_{j}$ with $\lip(f,u_{0})>q$. We can do this since the set $C_{X}\cap U_{j}$ is non-singleton as open balls contain nontrivial metric segments. Hence $\lip(f,u_{0},r)>q$ and $u_{0}\in U\cap T_{q,r}\neq\emptyset$.
\end{proof}
For the remainder of this section we work towards proving a version of Theorem~\ref{thm:convexresidual} for $\rho_{X}$-star-shaped subsets of weakly hyperbolic spaces. Namely, we  establish the following result:
\begin{theorem}\label{thm:starshapedresidual}
  Suppose that $C_{X}$ is separable and $C_{X}\subseteq \overline{B}(\stern(C_{X}),D_{X})$. Then there exists a $\sigma$-porous set $\widetilde{\N}\subseteq \M(C_{X},C_{Y})$ such that for $f\in \M(C_{X},C_{Y})\setminus\widetilde{\N}$, the set
  \begin{equation*}
    \widehat{R}(f)=\left\{x\in C_{X}\colon \globlip(f,x)=1\right\}
  \end{equation*}
  is a residual subset of $C_{X}$.
\end{theorem}

\begin{remark}
  Note that for contractive mappings in the sense of Rakotch the sets $R(f)$ and $\widehat{R}(f)$ coincide. Indeed, if $f$ is contractive in the sense of Rakotch, there exists a decreasing function $\varphi\colon (0,\infty)\to [0,1)$ such that $\rho_Y(f(x),f(y))\leq \varphi(\rho_X(x,y))\,\rho_X(x,y)$ for all distinct points $x,y\in C_X$. In other words
  \[
    \frac{\rho_Y(f(x),f(y))}{\rho_X(x,y)} \leq \varphi(\rho_X(x,y))
  \]
  for $x\neq y$, which shows that the expression on the left-hand side can only approach one when $y$ approaches $x$.
  With minor modifications, the proof of~\cite[Theorem~4]{Rei2005GenericityPorosity} shows that, if $X$ and $Y$ are
  hyperbolic spaces and $C_X\subseteq X$ and $C_Y\subseteq Y$ are non-empty, non-singleton, bounded, closed and $\rho_X$- 
  and $\rho_Y$-star-shaped subsets, respectively, there is a $\sigma$-porous subset $\bar{\mathcal{N}}\subset 
  \mathcal{M}(C_X,C_Y)$ such that all mappings in its complement are contractive in the sense of Rakotch.
\end{remark}
In view of the above remark, we can get the following corollary to Theorem~\ref{thm:starshapedresidual}, which is a 
strengthening of Theorems~\ref{thm:convexresidual} and~\ref{thm:starshapedresidual} restricted to the case where $X$ and $Y$ are hyperbolic spaces and $C_{X},C_{Y}$ are bounded. In particular, although we have seen that Lipschitz mappings on a star-shaped set~$C$ may not satisfy $\lip(f)=\sup_{x\in C}\lip(f,x)$, the following corollary indicates that typical nonexpansive mappings retain this property. 
\begin{corollary}
  Suppose $X$ and $Y$ are complete hyperbolic spaces, $C_{X}$ is separable and bounded and $C_Y$ is bounded. Then there exists a 
  $\sigma$-porous set $\widetilde{\mathcal{N}}\subseteq \mathcal{M}(C_{X},C_{Y})$ such that for every $f\in\mathcal{M}(C_{X},C_{Y})\setminus 
  \widetilde{\mathcal{N}}$, 
  the set
  \begin{equation*}
    R(f)=\left\{x\in C_{X}\colon \lip(f,x)=1\right\}
  \end{equation*} 
  is a residual subset of $C_{X}$.
\end{corollary}

For the proof of Theorem~\ref{thm:starshapedresidual}, we require an extension lemma for Lipschitz mappings. 
\begin{lemma}\label{lemma:fortsetzungvermutung}
  Let $(Z,d)$ and $(W,\rho)$ be metric spaces, $E\subseteq Z$ and $\Omega$ be a set such that $W\subseteq
  \ell_{\infty}(\Omega)$. Let $f:E\to W$ be a $1$-Lipschitz mapping, $u_{0}\in E$, $r>0$, $q\in(0,1)$, $q'\in(q,1)$
  and suppose that for every $x\in E\cap B(u_{0},r)$, we have
  \begin{equation*}
    \globlip(f,x)\leq q.
  \end{equation*}
  Then there exists a $1$-Lipschitz extension $F:Z\to\ell_{\infty}(\Omega)$ of $f$ and a number $s\in(0,r)$ 
  such that $\lip(F|_{B(u_{0},s)})\leq q '$.
\end{lemma}
\begin{proof}
  Using $W\subseteq \ell_{\infty}(\Omega)$, we view $f$ as a mapping from $E$ to $\ell_{\infty}(\Omega)$. Given $\omega\in\Omega$, a set $S\subseteq Z$ and a mapping $h\colon S\to \ell_{\infty}(\Omega)$ we let $h_{\omega}\colon S\to \R$ be defined by $h_{\omega}(x)=h(x)(\omega)$ for all $x\in S$. In what follows we will frequently use the identities
  \begin{equation}\label{eq:lipcstcomp-glob}
    \lip(h)=\sup\left\{\lip(h_{\omega})\colon \omega\in\Omega\right\},\qquad \widehat{\lip}(h,x)=\sup\left\{\widehat{\lip}(h_{\omega},x)\colon\omega\in\Omega\right\},
  \end{equation}
  which are easily derived from the definitions of the Lipschitz constants and the $\ell_{\infty}$~norm. We define the mapping $F: Z\to \ell_{\infty}(\Omega)$ componentwise by
  \begin{equation*}
    F_{\omega}(y):=\inf\left\{f_{\omega}(z)+\globlip(f_{\omega},z)d(z,y)\colon z\in E\right\},\quad y\in Z,\,\omega\in\Omega.
  \end{equation*}
  This mapping is a modification of the standard Lipschitz extension of $f$, as defined 
  in~\cite[Chapter~1]{BL2000GeomtericNonlinear}.  Let us verify that this mapping fulfills all the desired conditions. Firstly, we show that $F$ is an extension of $f$. Fix $\omega\in\Omega$. Letting $y\in E$ we observe from the definition that $F_{\omega}(y)\leq f_{\omega}(y)$. Moreover, given $\varepsilon>0$, we can choose $z\in E$ such that
  \begin{equation}\label{eq:zapp}
    F_{\omega}(y)\geq f_{\omega}(z)+\globlip(f_{\omega},z)d(z,y)-\varepsilon.
  \end{equation}
  This leads to the observation
  \begin{align*}
    F_{\omega}(y)&\geq f_{\omega}(z)+\globlip(f_{\omega},z)d(z,y)-\varepsilon\\
                 &\geq f_{\omega}(y)-\globlip(f_{\omega},z)d(z,y)+\globlip(f_{\omega},z)d(z,y)-\varepsilon\\
                 &=f_{\omega}(y)-\varepsilon.
  \end{align*}
  We conclude that $F_{\omega}(y)=f_{\omega}(y)$, as required.
  
  We now show that $F$ is $1$-Lipschitz. Let $\omega\in\Omega$ and $y_{1},y_{2}\in Z$. Given $\varepsilon>0$, we can choose $z_{2}\in E$ so that \eqref{eq:zapp} be satisfied with $y=y_{2}$ and $z=z_{2}$. From this we deduce
  \begin{align*}
    F_{\omega}(y_{1})-F_{\omega}(y_{2})&\leq (f_{\omega}(z_{2})+\globlip(f_{\omega},z_{2})d(z_{2},y_{1}))-(f_{\omega}(z_{2})+\globlip(f_{\omega},z_{2})d(z_{2},y_{2})-\varepsilon)\\
                                       &\leq\globlip(f_{\omega},z_2)d(y_{1},y_{2})+\varepsilon\leq d(y_{1},y_{2})+\varepsilon,
  \end{align*}
  where the final inequality uses $\lip(f_{\omega})\leq\lip(f)\leq1$.
  Similarly, we can show that $F_{\omega}(y_{2})-F_{\omega}(y_{1})\leq d(y_{1},y_{2})+\varepsilon$. We have shown that $\lip(F_{\omega})\leq 1$ for all $\omega\in\Omega$. Thus, by \eqref{eq:lipcstcomp-glob} we get that $\lip(F)\leq 1$. It only remains to verify that $F$ is locally a strict contraction around $u_{0}$. For this we will need the following claim. 
  \begin{claim}\label{lemma:dich}
    There exists $N>1$ such that for every $y\in B(u_{0},r/N)$, every $z\in E$ and every $\omega\in\Omega$, at least one of the following statements holds:
    \begin{enumerate}[(i)]
    \item $f_{\omega}(z)+\globlip(f_{\omega},z)d(z,y)> f_{\omega}(u_{0})+\globlip(f_{\omega},u_{0})d(u_{0},y)$.
    \item $\globlip(f_{\omega},z)\leq q'$.
    \end{enumerate}
  \end{claim}
  \begin{proof}
    We choose $N$ large enough so that 
    \begin{equation*}
      \frac{n+1}{n-1}\leq \frac{q'}{q}
    \end{equation*}
    for all $n\geq N$.  We set $s=r/N$ and fix $y\in B(u_{0},s)$ and $\omega\in\Omega$. If $z\in E\cap B(u_{0},r)$, then statement~(ii) already holds, because $\widehat{\lip}(f_{\omega},z)\leq \widehat{\lip}(f,z)\leq q<q'$, and there is nothing to prove. Therefore, we proceed by fixing a point $z\in E\setminus B(u_{0},r)$ and supposing that $z$ fails to satisfy the inequality of (i). In other words, we have
    \begin{equation}\label{eq:badz}
      f_{\omega}(z)+\globlip(f_{\omega},z)d(z,y)\leq f_{\omega}(u_{0})+\globlip(f_{\omega},u_{0})d(u_{0},y).
    \end{equation}
    We complete the proof by showing that statement (ii) holds for $z$. The left-hand side of~\eqref{eq:badz} can be bounded from below by the expression
    \begin{align*}
      f_{\omega}(u_{0})-q d(z,u_{0})&+\globlip(f_{\omega},z)(d(z,u_{0})-d(u_{0},y))\geq\\ &f_{\omega}(u_{0})-q d(z,u_{0})+\globlip(f_{\omega},z)(d(z,u_{0})-s).
    \end{align*}
    Moreover, we can bound the right-hand side of \eqref{eq:badz} from above by $f_{\omega}(u_{0})+q s$. We conclude from this that
    \begin{equation*}
      f_{\omega}(u_{0})-q d(z,u_{0})+\globlip(f_{\omega},z)(d(z,u_{0})-s)\leq f_{\omega}(u_{0})+q s.
    \end{equation*}
    Rearranging this inequality, we obtain
    \begin{equation*}
      \globlip(f_{\omega},z)\leq \frac{ q(d(z,u_{0})+s)}{d(z,u_{0})-s}=q\cdot\frac{n+1}{n-1},
    \end{equation*}
    where $n:=d(z,u_{0})/s\geq r/s= N$ and $d(z,u_0)-s\geq r-s>0$ since $z\not\in B(u_0,r)$. The last expression is bounded from above by $q'$.
  \end{proof}
  The proof of Lemma~\ref{lemma:fortsetzungvermutung} is now completed by proving the following claim:
  \begin{claim}
    Let $N$ be given by the statement of the previous claim. Then \[\lip(F|_{B(u_{0},r/N)})\leq q'.\]
  \end{claim}
  \begin{proof}
    Fix $y_{1},y_{2}\in B(u_{0},r/N)$ and $\omega\in\Omega$. Given $\varepsilon>0$, we can choose $z_{2}\in E$ such that~\eqref{eq:zapp} is satisfied with $y=y_{2}$, $z=z_{2}$ and 
    \begin{equation*}
      f_{\omega}(z_{2})+\globlip(f_{\omega},z_{2})d(z_{2},y_{2})\leq f_{\omega}(u_{0})+\globlip(f_{\omega},u_{0})d(u_{0},y_{2}).
    \end{equation*}
    Then by the first claim we have $\globlip(f_{\omega},z_{2})\leq q'$. We conclude that
    \begin{align*}
      F_{\omega}(y_{1})-F_{\omega}(y_{2})&\leq (f_{\omega}(z_{2})+\globlip(f_{\omega},z_{2})d(z_{2},y_{1}))-(f_{\omega}(z_{2})+\globlip(f_{\omega},z_{2})d(z_{2},y_{2})-\varepsilon)\\
                                         &\leq \globlip(f_{\omega},z_{2})d(y_{1},y_{2})+\varepsilon\leq q' d(y_{1},y_{2})+\varepsilon.
    \end{align*}
    Similarly, we can show that $F_{\omega}(y_{2})-F_{\omega}(y_{1})\leq q' d(y_{1},y_{2})+\varepsilon$. The above argument establishes that $\lip(f_{\omega}|_{B(u_{0},r/N)})\leq q'$ for every $\omega\in\Omega$. The conclusion of the claim follows.
  \end{proof}
  This completes the proof of Lemma~\ref{lemma:fortsetzungvermutung}.
\end{proof}
\begin{proof}[Proof of Theorem~\ref{thm:starshapedresidual}]
  Fix a countable dense subset $\Delta$ of $C_{X}$ and let $(U_{i})_{i=1}^{\infty}$ be an enumeration of all sets of 
  the form $B(x,r)$, where $x\in \Delta$ and $r\in\mathbb{Q}\cap(0,1)$ with $B(x,r)\subseteq B(\stern(C_{X}),D_{X})$. 
  By Theorem~\ref{thm:all}, each set $\Q(U_{i})$ is $\sigma$-porous.
  
  Suppose that $f\in\M(C_{X},C_{Y})$ is such that $\widehat{R}(f)$ is not residual. We complete the proof by showing 
  that $f\in\widetilde{\N}:=\bigcup_{i=1}^{\infty}\Q(U_{i})$. 
  
  From the assumption that $\widehat{R}(f)$ is not residual, we deduce that for some $q\in\mathbb{Q}\cap (0,1)$, 
  the open subset of $C_{X}$
  \begin{equation*}
    T_{q}:=\left\{x\in C_{X}\colon \globlip(f,x)>q\right\},
  \end{equation*}
  which occurs in the intersection in \eqref{eq:Rfbarint}, is not dense in $C_{X}$. Choose an open subset $U$ of 
  $X$ such $U\cap C_{X}\neq\emptyset$ and $U\cap T_{q}=\emptyset$. Then we have $\globlip(f,x)\leq q$ for all 
  $x\in C_{X}\cap U$. Using the inclusion $C_{X}\subseteq \overline{B}(\stern(C_{X}),D_{X})$ and the fact that the set 
  $\overline{B}(\stern(C_{X}),D_{X})\setminus B(\stern(C_{X}),D_{X})$ has empty interior in $C_{X}$, we can find 
  $u_{0}\in U\cap C_{X}\cap B(\stern(C_{X}),D_{X})$ and then choose $r>0$ such that $B(u_{0},r)\subseteq U\cap 
  B(\stern(C_{X}),D_{X})$. Applying Lemma~\ref{lemma:fortsetzungvermutung} with $E=C_{X}$, $Z=\conv(C_{X})$ and
  $W=C_{Y}$, we can find an extension $F:\conv(C_{X})\to \ell_{\infty}(\Omega)$ and an open ball $B(u_{0},s)\subseteq 
  B(u_{0},r)$ such that $\lip(F|_{B(u_{0},s)})<1$. 
  Choosing now $i\geq 1$ such that $U_{i}\subseteq B(u_{0},s)$, we have $\lip(F|_{U_{i}})<1$ and $f\in\Q(U_{i})$.
\end{proof}

\begin{remark}
  In the case where at least one of the sets $C_X$ and $C_Y$ is bounded, a more natural metric on 
  $\mathcal{M}(C_X,C_Y)$ is the metric of uniform convergence. More generally, we consider the space
  \[
  \mathcal{M}_{B}(C_X,C_Y) := \{f\colon C_X\to C_Y \colon \lip(f)\leq 1 \text{ and } f \text{ is bounded}\}
  \]
  of bounded mappings, that is, mappings where $f(C_X)\subset C_Y$ is bounded, and equip it with the metric
  \[
  d_\infty(f,g) := \sup \left\{d(f(x),g(x))\colon x\in C_X\right\}
  \]
  of uniform convergence.
  
  If the set $C_X$ is bounded, then $(\mathcal{M}(C_X,C_Y),d_\theta)$ and $(\mathcal{M}_B(C_X,C_Y),d_\infty)$ coincide as topological spaces. The inequalities
  \[
  \frac{\rho_Y(f(x),g(x))}{1+\rho_X(x,\theta)} \leq \rho_Y(f(x),g(x)) \leq (1+\diam(C_X)) \frac{\rho_Y(f(x),g(x))}{1+\rho_X(x,\theta)}
  \]
  show that in this case the metrics $d_\theta$ and $d_\infty$ are even Lipschitz equivalent. 

  With a small modification of the proof of Theorem~\ref{thm:all} we can also show that under the same assumptions, 
  the set
  \begin{equation*}
    \Q_B(U)=\left\{f\in \M_B(C_{X},C_{Y})\colon \inf_{F\in\E(f)}\lip(F|_{U})<1\right\}
  \end{equation*}
  is a $\sigma$-porous subset of $\mathcal{M}_B(C_X,C_Y)$.
  Since Theorem~\ref{thm:all} is the basis for the other porosity results in this section, we may deduce that
  the set $\mathcal{N}_B(C_{X},C_Y)$ of bounded strict contractions is a $\sigma$-porous subset of $\mathcal{M}_B(C_X,C_Y)$
  and that, in the separable setting, typical bounded nonexpansive mappings attain the maximal Lipschitz constant $1$ at typical points of their domain. In other words, all theorems in this section remain valid, if we replace $\M(C_X,C_Y)$ by $\M_B(C_X,C_Y)$ and $\N(C_X,C_Y)$ by $\N_B(C_X,C_Y)$.

  Let us conclude this remark by commenting on the necessary modification of the proofs in Section~4.
  Since Lemma~\ref{lemma:pert} actually implies that the perturbed mapping is $\varepsilon$-close to the original one not only with respect to $d_\theta$ but also with respect to $d_{\infty}$, we only have to notice that starting with a bounded mapping also the perturbed mapping we obtain is bounded and that in $\M_B(C_X,C_Y)$ the inclusion $B_\infty(f,\alpha\varepsilon)\subset B_\theta (f,\alpha\varepsilon)$ holds for all $f\in\M_B(C_X,C_Y)$ and all $\alpha,\varepsilon>0$, in order to get the results for bounded mappings. 
\end{remark}

\section{Proof of Theorem~\ref*{thm:all}}
In the present section we prove Theorem~\ref{thm:all}. Let $X$, $Y$, $C_{X}$, $C_{Y}$, $\conv(C_{X})$, $\Omega$,
$\theta$, $\M(C_{X},C_{Y})$, $\N(C_{X},C_{Y})$ and $\E(f)$ satisfy Hypotheses~\ref{hypotheses}. For the reader's convenience, 
we repeat the statement of Theorem~\ref{thm:all}:
\begin{theorem*}
  Let $U$ be an open subset of $X$ with $U\cap C_{X}\neq\emptyset$ and $U\subseteq B(\stern(C_{X}),D_{X})$. Then the set
  \begin{equation*}
    \Q(U)=\left\{f\in \M(C_{X},C_{Y})\colon \inf_{F\in\E(f)}\lip(F|_{U})<1\right\}
  \end{equation*}
  is $\sigma$-porous in $\M(C_{X},C_{Y})$.
\end{theorem*}

Let $U\subseteq X$ satisfy the hypotheses of Theorem~\ref{thm:all}. From this point onwards we only 
work inside metric segments in the space $X$ of the form $[x,y]$, where $x,y\in X$ with $\rho_{X}(x,y)<D_{X}$. 
Such metric segments are well defined because $X$ satisfies condition~\eqref{localuniqueness} of 
Definition~\ref{def:weaklyhyperbolic}. In particular, for $x,y\subseteq X$ with $\rho_{X}(x,y)<D_{X}$ and 
$\lambda\in[0,1]$, the point $(1-\lambda)x\oplus \lambda y\in X$ is well defined. We adopt a similar approach when 
working with metric segments in the space $Y$. In what follows we often identify a metric segment $[x,y]$ with a 
real interval. In particular, we endow metric segments with the natural ordering they inherit when viewed as real intervals. 

Let $\G$ denote the collection of all metric segments of the form $[w_{0},w_{1}]\subseteq C_{X}\cap U$ for which there 
exists a point $x_{0}\in \stern(C_{X})$ such that $w_{0}\in B(x_{0},D_{X})$ and $w_{1}\in[w_{0},x_{0}]$ with $w_{0}<w_{1}<x_{0}$. 
Since $U\subseteq B(\stern(C_{X}),D_{X})$ and $U\cap C_{X}\neq\emptyset$, the collection $\G$ is not empty. In the case where
$C_{X}$ is convex, we note that every metric segment in $C_{X}\cap U$ contains a metric subsegment which belongs to 
$\mathcal{G}$. For numbers $a<b\in(0,1)$ and $p\geq 2$, we define a collection of subsets $\Q_{a,b}^{p}(U)$ of $\Q(U)$ by
\begin{equation*}
  \Q_{a,b}^{p}(U):=\left\{f\in\Q(U)\colon a<\sup_{\Gamma\in\G}\lip(f|_{\Gamma})\leq b,\, \inf_{F\in\E(f)}\lip(F|_{U})\leq 1-\frac{1}{p}\right\}.
\end{equation*}
The significance of the above decomposition of $\Q(U)$ is revealed in the following lemma. 
\begin{lemma}\label{lemma:Qabpporous}
  If $a,b\in(0,1)$ and $p\geq 2$ satisfy the condition
  \begin{equation}\label{eq:ConditionABUBStar}
    b-a<\frac{a}{48(p-1)},
  \end{equation}
  then the set $\Q_{a,b}^{p}(U)$ is porous in $\mathcal{M}(C_{X},C_{Y})$.
\end{lemma}
Let us begin working towards a proof of Lemma~\ref{lemma:Qabpporous}. The basic idea of the proof is to take a mapping $f\in\Q_{a,b}^{p}(U)$ and to peturb it slightly to produce a nearby mapping $g\in\M(C_{X},C_{Y})$, the distance of which from the set $\Q_{a,b}^{p}(U)$ is a relatively large proportion of its distance from $f$. In order to control the Lipschitz constant of the mapping we construct, we first extend $f$ to a mapping $F:\conv(C_{X})\to \ell_{\infty}(\Omega)$ witnessing the fact that $f\in \Q_{a,b}^{p}(U)$ and then transform $F$ to a mapping $G:\conv(C_{X})\to\ell_{\infty}(\Omega)$ satisfying $G(C_{X})\subseteq C_{Y}$. The desired mapping $g\in\M(C_{X},C_{Y})$ can then be defined as the restriction of $G$ to $C_{X}$. 

The star-shaped nature of the sets $C_{X}$ and $C_{Y}$ presents two natural means of manipulating the mapping $F:\conv(C_{X})\to \ell_{\infty}(\Omega)$ in such a way that the condition $F(C_{X})\subseteq C_{Y}$ is preserved. One approach is to apply a mapping of the form $x\mapsto (1-\lambda(x))x\oplus \lambda(x)x_{0}$ with $x_{0}\in\stern(C_{X})$ to the set $\conv(C_{X})$ before applying the mapping $F$. Alternatively, one can first apply the mapping $F$ and then apply a mapping of the form $y\mapsto (1-\lambda(y))y\oplus\lambda(y) y_{0}$, with $y_{0}\in \stern(C_{Y})$. The latter approach is slightly more difficult than the former because the convex combination $(1-\lambda)F(x)\oplus \lambda y_{0}$
is not defined for all $x\in \conv(C_{X})$. In the present section we use both the aforementioned 
transformations and the next lemma captures their required properties. Given a real valued mapping $\lambda$ on $X$ we denote by $\|\lambda\|_{\infty} := \sup\{|\lambda(x)|\colon x\in X\}$ its supremum norm.

\begin{lemma}\label{lemma:pert}
  Let $Z\in\left\{X,Y\right\}$, $\sigma\in(0,1)$, $u_{0}\in C_{X}$, $z_{0}\in C_{Z}$ and $\pi\colon\conv(C_{X})\to \ell_{\infty}(\Omega)$ be a nonexpansive mapping such that 
  $\pi(C_{X})\subseteq C_{Z}$ and $0<\rho_{Z}(\pi(u_{0}),z_{0})<D_{Z}$. Then there is a number $r_{0}>0$ such that the following statement holds: Let $r,\varepsilon\in(0,r_{0})$, $\lambda\colon X\to[0,1]$ be a Lipschitz mapping such that $\lambda(x)=0$ for all  $x\in X\setminus B(u_{0},r)$,
  \begin{equation*}
    \left\|\lambda\right\|_{\infty}\leq \varepsilon/2\rho_{Z}(\pi(u_{0}),z_{0})\text{ and }\lip(\lambda)\leq \sigma/\rho_{Z}(\pi(u_{0}),z_{0}),
  \end{equation*}
  and suppose that $\pi(\conv(C_{X})\cap B(u_{0},r))\subseteq B(z_{0},D_{Z})$ and that every point $x\in C_{X}\cap B(u_{0},r)$ 
  admits a unique metric segment $[\pi(x),z_{0}]\subseteq C_{Z}$. Let $\beta$ be the mapping into $\ell_{\infty}(\Omega)$ defined 
  in the case $Z=X$ by
  \begin{equation*}
    \beta(x):=
      (1-\lambda(x))\pi(x)\oplus \lambda(x)z_{0} \qquad \forall x\in \conv(C_{X}),
  \end{equation*}
  and in the case $Z=Y$ by
  \begin{equation*}
  \beta(x):=\begin{cases}
  (1-\lambda(x))\pi(x)\oplus_{Y} \lambda(x)z_{0} & \text{if }x\in C_{X}\cap B(u_{0},r),
  \\
  \pi(x) & \text{if }x\in \conv(C_{X})\setminus B(u_{0},r).
  \end{cases}
  \end{equation*}
  Then $\beta$ satisfies the following conditions:
  \begin{enumerate}[(i)]
  \item\label{SxtoSy} $\beta(C_{X})\subseteq C_{Z}$;
  \item\label{beta-p} $\rho_{Z}(\beta(x),\pi(x))\leq\varepsilon$ for all $x\in C_{X}$;
  \item\label{lipbeta3} $\lip(\beta)\leq \max\left\{1,(1+\sigma)\lip(\pi|_{B(u_{0}, r)})+2\sigma\right\}$.
  \end{enumerate}
\end{lemma}
\begin{proof}
  We define 
  \begin{equation*}
  r_{0}=\min\left\{\rho_{Z}(\pi(u_{0}),z_{0}),\delta_{Z}(\pi(u_{0}),z_{0},\sigma),\rho_{Z}(\pi(u_{0}),z_{0})\delta_{Z}(\pi(u_{0}),z_{0},\sigma)\right\}.
  \end{equation*}
  Let $r,\varepsilon\in(0,r_{0})$ and $\lambda\colon X\to[0,1]$ be given by the hypotheses of Lemma~\ref{lemma:pert}. We now 
  verify statements~\eqref{SxtoSy}-\eqref{lipbeta3}.
  
  Statement~\eqref{SxtoSy} is immediate from the definition of $\beta$, the condition that $[\pi(x),z_{0}]
  \subseteq C_{Z}$ for all $x\in C_{X}\cap B(u_{0},r)$ and the fact that $\pi(C_{X})\subseteq C_{Z}$. For statement~\eqref{beta-p} 
  we make the following observation: If $x\in \conv(C_{X})\setminus B(u_{0},r)$, then $\beta(x)=\pi(x)$. 
  Otherwise, we have
  \begin{align*}
    \rho_{Z}(\beta(x),\pi(x))&\leq \lambda(x)\rho_{Z}(\pi(x),z_{0})\leq\left\|\lambda\right\|_{\infty}
                               (\rho_{Z}(\pi(u_{0}),z_{0})+r)\leq \varepsilon,
  \end{align*}
  using $r<\rho_{Z}(\pi(u_{0}),z_{0})$ and $\left\|\lambda\right\|_{\infty}\leq\varepsilon/2\rho_{Z}(\pi(u_{0}),z_{0})$.  
  
  To prove~\eqref{lipbeta3}, we fix points $x,y$ in the intersection of the domain of $\beta$ with $B(u_{0},r)$ and observe that
  \begin{align}
    \rho_{Z}(\beta(x),\beta(y))&\leq\rho_{Z}((1-\lambda(x))\pi(x)\oplus\lambda(x)z_{0},(1-\lambda(x))\pi(y)\oplus\lambda(x)z_{0})\nonumber\\
                               &\qquad\qquad\qquad+\rho_{Z}((1-\lambda(x))\pi(y)\oplus\lambda(x) z_{0},
                                 (1-\lambda(y))\pi(y)\oplus\lambda(y)z_{0})\nonumber\\
                               &\leq(1+\sigma)\rho_{Z}(\pi(x),\pi(y))+\left|\lambda(y)-\lambda(x)\right|\rho_{Z}(\pi(y),z_{0})\nonumber\\
                               &\leq (1+\sigma)\lip(\pi|_{B(u_{0},r)})\rho_{X}(x,y)\nonumber\\
                               &\qquad\qquad\qquad\qquad+\lip(\lambda)(r+\rho_{Z}(\pi(u_{0}),z_{0}))\rho_{X}(x,y)\nonumber\\
                               &\leq ((1+\sigma)\lip(\pi|_{B(u_{0},r)})+2\sigma)\rho_{X}(x,y). \label{eq:finallipineq}
  \end{align}
  In deriving the above inequalities we used the definition of $r_{0}$ and the constraints on $r,\varepsilon$ and $\lambda$ to deduce that $0\leq\lambda(x)< \delta_{Z}(\pi(u_{0}),z_{0},\sigma)$ and $\rho_{Z}(\pi(x),\pi(u_{0})),\rho_{Z}(\pi(y),\pi(u_{0}))< \delta_{Z}(\pi(u_{0}),z_{0},\sigma)$. These conditions allow us to apply condition~\eqref{thinishtriangles} of Definition~\ref{def:weaklyhyperbolic} to obtain the second inequality in the sequence above. Note that the above inequalities remain true for $x\in\partial B(u_0,r)$ when, for $z\in \ell_{\infty}(\Omega)$, we interpret the expression $(1-\lambda(x))z\oplus \lambda(x)z_{0}$ as $z$ since in that case $\lambda(x)=0$ and $\lip(\pi|_{B(u_0,r)})=\lip(\pi|_{\overline{B}(u_0,r)})$.
  
  Having established~\eqref{eq:finallipineq} and noting that $\beta$ coincides with the nonexpansive mapping $\pi$ outside of $B(u_{0},r)$, we only need to verify the Lipschitz bound for the quantity $\rho_{Z}(x,y)$ for points $x,y$ in the domain of $\beta$ with $x\in B(u_{0},r)$ and $y\notin B(u_{0},r)$. Such points admit a metric segment $[x,y]$ in $\conv(C_{X})$ and an application of the Intermediate Value Theorem provides a point $x'\in[x,y]$ with $\rho_{X}(x',u_{0})=r$, so that $x'\in\partial B(u_{0},r)$. Using the Lipschitz bound derived above for points $u,v$ in the domain of $\beta$ with $u\in\partial B(u_{0},r)$ and $v\in B(u_{0},r)$, we may now deduce that
  \begin{align*}
    \rho_{Z}(\beta(x),\beta(y)) &\leq\rho_{Z}(\beta(x),\beta(x'))+\rho_{Z}(\beta(x'),\beta(y))\\
                               &\leq ((1+\sigma)\lip(\pi|_{B(u_{0},r)})+2\sigma) \rho_{X}(x,x')+\rho_{X}(x',y) \\
                               &\leq \max\{1,((1+\sigma)\lip(\pi|_{B(u_{0},r)})+2\sigma)\} (\rho_{X}(x,x')+\rho_X(x',y))\\
                               &= \max\{1,((1+\sigma)\lip(\pi|_{B(u_{0},r)})+2\sigma)\}(\rho_{X}(x,y).
  \end{align*}
  This completes the proof of \eqref{lipbeta3} and of Lemma~\ref{lemma:pert} itself.  
\end{proof}

Fix a mapping $f\in \Q_{a,b}^{p}(U)$ and choose a metric segment $\Gamma=[w_{0},w_{1}]\in\G$ such that 
$a<\lip(f|_{\Gamma})\leq b$ and an extension $F:\conv(C_{X})\to \ell_{\infty}(\Omega)$ of $f$ such that 
$\lip(F|_{U})\leq 1-\frac{1}{p}$. Choose $x_{0}\in \stern(C_{X})$ such that $w_{0}\in B(x_{0},D_{X})$ and 
$[w_{0},w_{1}]\subseteq [w_{0},x_{0}]$ with $w_{0}<w_{1}<x_{0}$. The mapping $F$ coincides with $f$ on the
segment $\Gamma$. Therefore we have $a<\lip(F|_{\Gamma})\leq b$. Applying Lemma~\ref{lem:LipOnGamma} 
with $C=(w_{0},w_{1})\subseteq [w_{0},x_{0}]$, we find a point $u_{0}\in (w_{0},w_{1})$ such that
\begin{equation*}
  \liminf_{t\to 0^{+}}\frac{\rho_{Y}(F((1-t)u_0\oplus tx_0), F(u_0))}{t\rho_{X}(u_0,x_0)}>a.
\end{equation*}

Choose $\sigma\in(0,1)$ such that $(1-\frac{1}{p})(1+3\sigma)\leq 1$. Let $r_{0}$ be given by the conclusion of Lemma~\ref{lemma:pert} applied to $Z=X$, $\sigma$, $u_{0}$, $z_{0}=x_{0}$ and $\pi=\operatorname{id}_{\conv(C_{X})}\colon \conv(C_{X})\to\ell_{\infty}(\Omega)$. Let $r\in(0,r_{0})$ be small enough so that $B(u_{0},3r)\subseteq U\cap B(x_{0},D_{X})$. Using $u_{0}<w_{1}<x_{0}$, we may choose $\varepsilon_{0}\in(0,\min\left\{\sigma r/2,\rho_{X}(u_{0},x_{0})/2,1\right\})$ small enough so that 
\begin{equation}\label{eq:liminfu0}
  (1-t)u_{0}\oplus tx_{0}\in [w_{0},w_{1}]=\Gamma\quad\text{and}\quad\frac{\rho_{Y}(F((1-t)u_0\oplus t x_0),F(u_0))}{t\rho_{X}(u_0,x_0)} > a
\end{equation}
for all $t\in (0,2\varepsilon_{0}/\rho_{X}(u_{0},x_{0}))$. Fixing $\varepsilon\in(0,\varepsilon_0)$, we introduce the mappings
\[
\psi\colon X\to [0,1], \quad x\mapsto 
\begin{cases} 
  1-\frac{2}{r}\dist\left(x, B\left(u_0,\frac{r}{2}\right)\right) & x\in B(u_0,r)\\
  0 & x\not\in B(u_0,r)
\end{cases}
\]
and
\[
\varphi\colon \mathbb{R}\to\mathbb{R}, \quad t\mapsto \min\left\{|t|,\frac{\varepsilon}{\sigma}\right\}.
\]
These mappings satisfy 
\[
\lip\psi = \frac{2}{r}, \quad \|\psi\|_\infty=1, \quad \lip\varphi =1\quad\text{and}\quad
\|\varphi\|_\infty = \frac{\varepsilon}{\sigma}.
\]

Since the metric segment $[u_0,x_0]$ is isometric to a closed real interval, it is an absolute $1$-Lipschitz 
retract by Proposition~1.4 in~\cite[p.~13]{BL2000GeomtericNonlinear}. Let 
$R\colon X\to [u_0,x_0]$ be a $1$-Lipschitz retraction and $c\colon [0,\rho_{X}(u_{0},x_{0})]\to [u_0,x_0]$ 
be a metric embedding with $c(0)=u_0$. We define
\[
q \colon C_{X} \to [0,\rho_{X}(u_{0},x_{0})], \quad x \mapsto c^{-1}(R(x)). 
\]
Since $q$ is the composition of $1$-Lipschitz mappings, it is also a $1$-Lipschitz mapping.
Finally, we also define the mapping
\[
\lambda\colon X\to\mathbb[0,1], \quad x\mapsto \frac{\sigma}{2\rho_{X}(u_{0},x_{0})}\psi(x)\varphi(q(x)).
\]
This mapping satisfies $\lambda(x)=0$ whenever $x\in X\setminus B(u_{0},r)$, $\left\|\lambda\right\|_{\infty}\leq 
\varepsilon/2\rho_{X}(u_{0},x_{0})$ and
\begin{align*}
  \lip(\lambda) &\leq \frac{\sigma}{2\rho_{X}(u_{0},x_{0})}\left(\lip(\varphi)\|\psi\|_\infty+\lip(\psi)\|\varphi\|_\infty\right)
                  = \frac{1}{2\rho_{X}(u_{0},x_{0})} \left(\sigma +\frac{2}{r} \varepsilon\right)\\
                &\leq \frac{1}{2\rho_{X}(u_{0},x_{0})}\left(\sigma + \sigma\right)
                  \leq \frac{1}{2\rho_{X}(u_{0},x_{0})}2\sigma = \frac{\sigma}{\rho_{X}(u_{0},x_{0})}
\end{align*}
because $\varepsilon < \sigma \frac{r}{2}$. We observe now that the conditions of Lemma~\ref{lemma:pert} are satisfied
for $Z=X$, $\sigma$, $u_{0}$, $z_{0}=x_{0}$ $\pi=\operatorname{id}_{\conv(C_{X})}$ $r,\varepsilon\in(0,r_{0})$ and $\lambda$. Finally also note that
\begin{equation}\label{eq:lambau0zero}
  \lambda(u_0)=0
\end{equation}
since $u_0\in[u_0,x_0]$ implies $R(u_0)=u_0$, $q(u_0)=0$ and hence $\varphi(u_0)=0$ .
Applying Lemma~\ref{lemma:pert}, we conclude that the mapping $\beta\colon \conv(C_{X})\to\conv(C_{X})$, defined by
\[
\beta(x):=(1-\lambda(x))x\oplus\lambda(x) x_0, 
\]
satisfies $\beta(C_{X})\subseteq C_{X}$, $\rho_{X}(\beta(x),x)\leq\varepsilon$ for all $x\in C_{X}$ and $\lip(\beta)\leq 1+3\sigma$.
\begin{lemma}\label{lem:PropertiesGStar}
  The mapping 
  \[
  G\colon \conv(C_{X}) \to \ell_{\infty}(\Omega), \quad x\mapsto F(\beta(x))
  \]
  satisfies the following conditions:
  \begin{enumerate}[(i)]
  \item\label{G:StoY} $G(C_{X})\subseteq C_{Y}$;
  \item\label{eq:DistFandGStar} $\rho_{Y}(F(x),G(x))\leq\varepsilon$ for all  $x\in C_{X}$;
  \item\label{eq:LipGStar} $\lip(G)\leq 1$;
  \item\label{Gsteep} For $s=\varepsilon/\rho_{X}(u_{0},x_{0})$, we have $(1-s)u_{0}\oplus sx_{0}\in \Gamma$ and
    \begin{equation*}
      \frac{\rho_{Y}(G((1-s)u_0\oplus sx_0),G(u_0))}{s\rho_{X}(u_0,x_0)} > a\left(1+\frac{\sigma}{4}\right).
    \end{equation*}
  \end{enumerate}
\end{lemma}
\begin{proof} 
 The inclusion $\beta(C_{X})\subseteq C_{X}$ together with the fact that $F$ is an extension of the mapping $f:C_{X}\to C_{Y}$ implies condition~\eqref{G:StoY}. Condition~\eqref{eq:DistFandGStar} follows immediately from the fact that $\rho_{X}(\beta(x),x)\leq\varepsilon$ for all $x\in C_{X}$. Let us now verify condition~\eqref{eq:LipGStar}: Since $G$ coincides with $F$ outside of $B(u_{0},r)$ and is defined on a $\rho_{X}$-convex set, an argument similar to the one at the end of the proof of Lemma~\ref{lemma:pert} shows that it suffices to prove $\lip(G|_{B(u_{0},r)})\leq 1$.
 If we show $\beta(B(u_0,r))\subseteq U$, this inequality follows from  $\lip(\beta)\leq 1+3\sigma$, $\lip(F|_{U})\leq (1-\frac{1}{p})$ and $(1+3\sigma)(1-\frac{1}{p})\leq 1$. In order to show the required inclusion, we use $\rho_{X}(\beta(x),x)\leq \varepsilon$ and $\varepsilon<r$ to get that $\beta(B(u_{0},r))\subseteq \overline{B}(u_{0},r+\varepsilon)\subseteq B(u_{0},3r)\subseteq U$. 

 Next we turn our attention to~\eqref{Gsteep}. The choice of $\varepsilon_{0}$ and 
 $s=\varepsilon/\rho_{X}(u_{0},x_{0})<2\varepsilon_{0}/\rho_{X}(u_{0},x_{0})$ imply that $(1-s)u_{0}\oplus sx_{0}\in\Gamma$. 
 For $t\in(0,1)$, we define 
 \begin{equation*}
   \gamma(t):=(1-t)[(1-s)u_{0}\oplus sx_{0}]\oplus tx_{0}.
 \end{equation*}
 Using condition~\eqref{closedwrtsubseg} of Definition~\ref{def:weaklyhyperbolic} in the weakly hyperbolic space $X$, we note that $\gamma(t)$ lies on the metric segment $[u_{0},x_{0}]$ in between $(1-s)u_{0}\oplus sx_{0}$ and $x_{0}$. Therefore we can compute $\rho_{X}(\gamma(t),u_{0})$ as the sum
 \begin{align*}
   \rho_{X}(\gamma(t),u_{0})&=\rho_{X}(\gamma(t),(1-s)u_{0}\oplus sx_{0})+\rho_{X}((1-s)u_{0}\oplus sx_{0},u_{0}))\nonumber\\
                            &=t(1-s)\rho_{X}(u_{0},x_{0})+s\rho_{X}(u_{0},x_{0})\nonumber\\
                            &=
(t+s(1-t))\rho_X(u_{0},x_{0}).
 \end{align*}
 It follows that 
 \begin{equation}\label{eq:distgammat}
   \gamma(t)=(1-\alpha(t))u_{0}\oplus \alpha (t)x_{0},\qquad\text{ where }\alpha(t):=t+s(1-t).
 \end{equation}
 Using the definitions of the mappings $\varphi$, $q$ and $\psi$ together with 
 \[
   \rho_{X}(u_0,(1-s)u_0\oplus sx_0)=s\rho_{X}(u_0,x_0)= \varepsilon<\varepsilon/\sigma<r/2
 \]
 we obtain $\varphi(q((1-s)u_{0}\oplus sx_{0}))=\varepsilon$, $\psi((1-s)u_{0}\oplus sx_{0})=1$ and subsequently,
 \begin{equation*}
   \lambda((1-s)u_{0}\oplus sx_{0})=\sigma\varepsilon/2\rho_{X}(u_{0},x_{0})=\sigma s/2.
 \end{equation*}
 We conclude that $\beta((1-s)u_{0}\oplus sx_{0})=\gamma(\frac{\sigma s}{2})$. From~\eqref{eq:distgammat} we see that $\alpha(\frac{\sigma s}{2})<2s<2\varepsilon_{0}/\rho_{X}(u_{0},x_{0})$. Therefore we can apply~\eqref{eq:liminfu0} to deduce
 \begin{align*}
   \frac{\rho_{Y}(G((1-s)u_{0}\oplus sx_{0}),G(u_{0}))}{s\rho_{X}(u_{0},x_{0})}
   &=\frac{\rho_{Y}\big(F((1-\alpha(\frac{\sigma s}{2}))u_{0}\oplus\alpha(\frac{\sigma s}{2})x_{0}),F(u_{0})\big)}
     {\alpha(\frac{\sigma s}{2})\rho_{X}(u_{0},x_{0})}\frac{\alpha(\frac{\sigma s}{2})}{s}\\
   &>a\left(\frac{\sigma}{2}+1-\frac{\sigma s}{2}\right)>a\left(1+\frac{\sigma}{4}\right).
 \end{align*} 
 Above we used~\eqref{eq:lambau0zero} to get $G(u_0)=F(\beta(u_0))=F(u_0)$ in the first line and the condition $s<\varepsilon_{0}/\rho_{X}(u_{0},x_{0})<1/2$ to get the final inequality.
\end{proof}
We are now ready to prove Lemma~\ref{lemma:Qabpporous}. 
\begin{proof}[Proof of Lemma~\ref{lemma:Qabpporous}]
  Fix $f\in\Q_{a,b}^{p}(U)$ and let $\Gamma\in\G$, $F:\conv(C_{X})\to\ell_{\infty}(\Omega)$,  $u_{0}\in\Gamma$, $\sigma\in(0,1)$ satisfying $(1+3\sigma)(1-\frac{1}{p})\leq 1$ and $\varepsilon_{0}>0$ be defined according 
  to the above construction. The precise value of $\sigma$ will be determined at the end of this proof. Given $\varepsilon\in(0,\varepsilon_{0})$, let the mapping $G\colon \conv(C_{X})\to \ell_{\infty}(\Omega)$ be given by the statement of Lemma~\ref{lem:PropertiesGStar}. Define $g\colon C_{X}\to C_{Y}$ to be the restriction of $G$ to the set $C_{X}$. From Lemma~\ref{lem:PropertiesGStar} it is clear that $g\in \M(C_{X},C_{Y})$ with $d_{\theta}(g,f)\leq\varepsilon$. We complete the proof by showing that
  \begin{equation*}
    B_{\theta}\left(g,\frac{a\sigma}{32(1+\rho_{X}(u_{0},\theta))}\varepsilon\right)\cap \Q_{a,b}^{p}(U)=\emptyset.
  \end{equation*}
  Let $h\in B_{\theta}\left(g,\frac{a\sigma}{32(1+\rho_{X}(u_{0},\theta))}\varepsilon\right)$. Then
  \[
  \rho_{Y}(g(x),h(x)) \leq \frac{1+\rho_{X}(x,\theta)}{1+\rho_{X}(u_0,\theta)} \frac{a\sigma}{32}\varepsilon
  \leq \frac{a\sigma}{16}\varepsilon
  \]
  for $x\in C_{X}\cap B(u_0,1)$ and, in particular,
  \[
  \rho_{Y}(g((1-s)u_0\oplus s x_0),h((1-s)u_0\oplus s x_0)) \leq \frac{a\sigma}{16}\varepsilon
  \]
  for $s=\varepsilon/\rho_{X}(u_{0},x_{0})$ because $\varepsilon<\varepsilon_{0}<1$.
  Therefore, using Lemma~\ref{lem:PropertiesGStar}, part \eqref{Gsteep} and the fact that $g$ coincides with $G$ on the segment $[u_{0},(1-s)u_{0}\oplus sx_{0}]\subseteq\Gamma\subseteq C_{X}$, we deduce that
  \begin{align*}
    \frac{\rho_{Y}(h((1-s)u_0\oplus s x_0),h(u_0))}{s\rho_{X}(u_0,x_0)}
    &\geq \frac{\rho_{Y}(g((1-s)u_0\oplus s x_0),g(u_0))}{s\rho_{X}(u_0,x_0)} - 2 \frac{a\sigma}{16}\frac{\varepsilon}{s\rho_{X}(u_0,x_0)}\\
    & > a \left(1+\frac{\sigma}{4}\right)-\frac{a\sigma}{8}= a\left(1+\frac{\sigma}{8}\right).
  \end{align*}
  We conclude from the above inequalities that $\lip(h|_{\Gamma})>b$, when we choose $\sigma=\frac{16(b-a)}{a}$. Condition~\eqref{eq:ConditionABUBStar} ensures that such a choice of $\sigma$ satisfies $(1+3\sigma)(1-\frac{1}{p})\leq 1$, as required. This establishes $h\notin \Q_{a,b}^{p}(U)$ and completes the proof.
\end{proof}
The sets $\Q_{a,b}^{p}(U)$ do not quite cover the whole of the set $\Q(U)$. In the next lemma, we verify that the elusive mappings in $\Q(U)$ form a porous subset of $\M(C_{X},C_{Y})$.
\begin{lemma}\label{lemma:ConstPorousStar}
  The set
  \[
  \Q_{0}(U) := \left\{f\in\Q(U)\colon \sup_{\Gamma\in \G}\lip(f|_{\Gamma})=0,\, \inf_{F\in\E(f)}\lip(F|_{U})<1\right\}
  \]
  is porous in $\mathcal{M}(C_{X},C_{Y})$.
\end{lemma}
\begin{proof}
  Fix a mapping $f\in\Q_{0}(U)$ and choose an extension $F:\conv(C_{X})\to\ell_{\infty}(\Omega)$ of $f$ with $\lip(F|_{U})<1$. Choose $x_{0}\in\stern(C_{X})$ such that $U\cap B(x_{0},D_{X})\neq\emptyset$ and set $U'=U\cap B(x_{0},D_{X})\setminus\left\{x_{0}\right\}$. We make the following claim:
  \begin{claim}
    There exist $u_{0}\in C_{X}\cap U'$, $y_{0}\in C_{Y}\setminus \left\{f(u_{0})\right\}$ and $r>0$ such that $F(\conv(C_{X})\cap B(u_{0},r))\subseteq B(y_{0},D_{Y})$ and for every $x\in C_{X}\cap B(u_{0},r)$, there is a unique metric segment $[f(x),y_{0}]\subseteq C_{Y}$ .
  \end{claim}
  \begin{proof}
    We distinguish between two cases. First assume that $f(x)\in\stern(C_{Y})$ for all $x\in C_{X}\cap U'$. Then we choose $u_{0}\in C_{X}\cap U'$ arbitrarily and let $r>0$ be small enough so that $B(u_{0},r)\subseteq U'$ and $F(\conv(C_{X})\cap B(u_{0},r))\subseteq B(f(u_{0}),D_{Y}/2)$. Let $y_{0}\in C_{Y}\cap B(f(u_{0}),D_{Y}/2)\setminus\left\{f(u_{0})\right\}$ be arbitrary. The assertion of the claim is now clear.
    
    In the remaining case we choose $u_{0}\in C_{X}\cap U'$ such that $f(u_{0})\notin \stern(C_{Y})$ and use the fact that $C_{Y}\subseteq B(\stern(C_{Y}),D_{Y})$ to choose $y_{0}\in \stern(C_{Y})\cap B(f(u_{0}),D_{Y})$. Letting $r>0$ be sufficiently small so that $F(\conv(C_{X})\cap B(u_{0},r))\subseteq B(y_{0},D_{Y})$, we verify the claim.
  \end{proof}
  Let $u_{0}\in C_{X}\cap U'$, $y_{0}\in C_{Y}\setminus\left\{f(u_{0})\right\}$ and $r>0$ be given by the claim. Choose $\sigma\in(0,1)$ small enough so that
  \begin{equation*}
    (1+\sigma)\lip(F|_{U})+2\sigma\leq 1.
  \end{equation*}
  By making $r$ smaller if necessary we may assume that $B(u_{0},r)\subseteq U'$ and $r\in(0,r_{0})$, where $r_{0}>0$ is given by the conclusion of Lemma~\ref{lemma:pert} with $Z=Y$, $\sigma$, $u_{0}$, $z_{0}=y_{0}$ and $\pi=F$. Set $\varepsilon_{0}=r$.
  Given $\varepsilon\in(0,\varepsilon_{0})$, we define a mapping $\lambda:X\to [0,1]$ by
  \begin{equation*}
    \lambda(x):=\frac{\sigma}{2\rho_{Y}(f(u_{0}),y_{0})}\max\left\{\varepsilon-\rho_{X}(x,u_{0}),0\right\}, \quad x\in X.
  \end{equation*}
  Then, 
  \[
    \lambda(x)=0 \text{ for all } X\setminus B(u_{0},r), \;\left\|\lambda\right\|_{\infty}\leq \varepsilon/2\rho_{Y}(f(u_{0}),y_{0})\text{ and }\lip(\lambda)\leq \sigma/\rho_{Y}(f(u_{0}),y_{0}).
  \]
  Thus, the conditions of Lemma~\ref{lemma:pert} are satisfied for $Z=Y$, $\sigma$, $u_{0}$, $z_{0}=y_{0}$ $\pi=F$, $r,\varepsilon\in(0,r_{0})$ and $\lambda$. Therefore, Lemma~\ref{lemma:pert} asserts that the mapping $G$ defined by
  \begin{equation*}
    G(x):=\begin{cases}
    (1-\lambda(x))F(x)\oplus\lambda(x)y_{0} & \text{if }x\in C_{X}\cap B(u_{0},r),\\
    F(x) & \text{if }x\in\conv(C_{X})\setminus B(u_{0},r),
    \end{cases}
  \end{equation*}
  satisfies $G(C_{X})\subseteq C_{Y}$, $\rho_{Y}(G(x),F(x))\leq\varepsilon$ for all $x\in C_{X}$ and $\lip(G)\leq1$. Clearly, the restriction $g$ of the mapping $G$ to the set $C_{X}$ can be viewed as an element of $\M(C_{X},C_{Y})$ satisfying $d_{\theta}(g,f)\leq\varepsilon$. 
  
  Since $B(u_{0},r)\subseteq U'=U\cap B(x_{0},D_{X})\setminus\left\{x_{0}\right\}$ and $x_{0}\in \stern(C_{X})$, we have that $[u_{0},x_{0}]\subseteq C_{X}$. Identifying the metric segment $[u_{0},x_{0}]$ with a real interval we have the $u_{0}<u_{0}+\varepsilon<u_{0}+r<x_{0}$. Hence $[u_{0},u_{0}+\varepsilon]\in \mathcal{G}$. Using $\lambda(u_0+\varepsilon)=0$ and the fact that $f$ is constant on the segment $[u_{0},u_{0}+\varepsilon]$, we get
  \begin{align*}
    \rho_{Y}(g(u_{0}+\varepsilon),g(u_{0}))
    &=\rho_{Y}(f(u_{0}),(1-\frac{\sigma\varepsilon}{2\rho_{Y}(f(u_{0}),y_{0})})f(u_{0})\oplus 
      \frac{\sigma\varepsilon}{2\rho_{Y}(f(u_{0}),y_{0})}y_{0})=\frac{\sigma\varepsilon}{2}.
  \end{align*}
  For all $h\in\M(C_{X},C_{Y})$ with 
  \[
  d_{\theta}(h,g)\leq \frac{\sigma\varepsilon}{6(1+\rho_{X}(u_{0},\theta)+\varepsilon_{0})},
  \]
  we have $\rho_{Y}(h(x),g(x))\leq\sigma\varepsilon/6$ for $x=u_{0},u_{0}+\varepsilon$ which, when combined with the above equation, implies that $h$ is non-constant on the metric segment $[u_{0},u_{0}+\varepsilon]\in\G$. Hence 
  \[
  B(g,\frac{\sigma\varepsilon}{6(1+\rho_{X}(u_{0},\theta)+\varepsilon_{0})})\cap \Q_{0}(U)=\emptyset
  \]
  and the proof is complete.
\end{proof}
\begin{remark}
  \begin{enumerate}[(i)]
  \item The proof of Lemma~4.4 is the only place in the proof of Theorem~\ref{thm:all}, or indeed any of the results of Section~3, where we use the hypothesis that $C_{Y}$ is $\rho_Y$-star-shaped and satisfies $C_{Y}\subseteq B(\stern(C_{Y}),D_{Y})$.
  \item In the special case where $C_{X}$ is $\rho_{X}$-convex, the set $Q_{0}(U)$ becomes simply the set of all mappings $f\in \Q(U)$ which are constant on the set $C_{X}\cap U$. The conclusion of Lemma~\ref{lemma:ConstPorousStar} is then valid under much weaker assumptions on the set $C_{Y}$. For example, it suffices to assume that $C_{Y}$ is a metric space in which every point belongs to some non-trivial geodesic. Thus, if we restrict our attention to the case where $C_{X}$ is $\rho_{X}$-convex, the results of Section~3 can be generalised accordingly.
  \end{enumerate}
\end{remark}
\begin{proof}[Proof of Theorem~\ref{thm:all}]
  For each $f\in\Q(U)\setminus \Q_{0}(U)$, we have
  \begin{equation*}
    (\sup_{\Gamma\in \G}\lip(f|_{\Gamma}),\inf_{F\in\E(f)}\lip(F|_{U}))\in(0,1)^{2}.
  \end{equation*}
  The family of all rectangles of the form $(a,b)\times(0,1-\frac{1}{p})$, where $p\in\mathbb{N}$ with $p\geq 2$ and $0<a<b<1$ satisfy~\eqref{eq:ConditionABUBStar}, is an open cover of $(0,1)^{2}$. Therefore, since $(0,1)^{2}$ is a Lindelöf space, this family admits a countable subcover $((a_{i},b_{i})\times(0,1-\frac{1}{p_{i}}))_{i=1}^{\infty}$. Hence we may write
  \[
  \Q(U) = \bigcup_{i=1}^{\infty} \Q_{a_{i},b_{i}}^{p_{i}}(U) \cup \Q_0(U).
  \]
  Applying now Lemma~\ref{lemma:ConstPorousStar} and Lemma~\ref{lemma:Qabpporous}, we arrive at the asserted result.
\end{proof}

\section{An application to set-valued mappings}

The goal of this section is to examine properties of spaces of non-empty, closed and bounded subsets of hyperbolic spaces in order to show that these spaces can be chosen as the range of the nonexpansive mappings in the theorems which were established in the previous sections.

Let $(X,\rho)$ be a complete hyperbolic space and $C\subseteq X$ be a non-empty, non-singleton, closed and $\rho$-star-shaped set. We consider the space
\[
\mathcal{B}(C) := \left\{A\subseteq C\colon\; A \text{ is nonempty, closed and bounded}\right\}
\]
equipped with the Pompeiu-Hausdorff metric
\[
h(A,B) := \max\big\{\sup\{\dist(a,B)\colon a\in A\},\; \sup\{\dist(b,A)\colon b\in B \}\big\},
\]
where $\dist(x,A) := \inf\{\rho(x,a)\colon a\in A\}$. 
The space $\mathcal{B}(C)$ is a complete metric space by~\cite[§33, IV]{Kur1966Topology}. In addition to the hyperspace of all bounded and closed sets, we also consider the subspaces $\mathcal{K}(C)$ of compact subsets and $\mathcal{CB}(C)$ of $\rho$-convex, bounded and closed sets.

In the case where $X$ is a Banach space, the following lemma is a consequence of Proposition~4.6 in~\cite{Str2014Porosity}.

\begin{lemma}\label{lem:starShapedSetValued}
  There is a family $\mathcal{F}$ of metric segments in $\mathcal{B}(C)$ such that the triple $(\mathcal{B}(C),h,\mathcal{F})$ is a space of temperate curvature with $D_{\mathcal{B}(C)}=\infty$ and $\mathcal{B}(C)$ is a $h$-star-shaped subset of this space.
\end{lemma}

\begin{proof}
  For $A\in\mathcal{B}(C)$, $A\neq\{c\}$, we define 
  \begin{equation}\label{eq:SetConvex}
  A^{(1-\lambda)} := \{(1-\lambda) a \oplus \lambda c\colon a\in A\} \quad\text{and}\quad
  (1-\lambda) A\oplus \lambda\{c\} := \overline{A^{(1-\lambda)}},
  \end{equation}
  and set
  \[
  \mathcal{F}:=\left\{\left\{(1-\lambda) A\oplus \lambda\{c\}\colon \lambda\in[0,1]\right\}\colon A\in\mathcal{B}(C), \; c\in\stern(C)\right\}. 
  \]
  In order to show that $\mathcal{F}$ is a well-defined collection of metric segments in $\mathcal{B}(C)$, we have to show that $(1-\lambda) A\oplus \lambda\{c\}\in\mathcal{B}(C)$ for every $A\in\mathcal{B}(C)$, $A\neq\{c\}$, and that the sets $[A,\left\{c\right\}]:=\{(1-\lambda) A\oplus \lambda\{c\}\in\mathcal{B}(C)\colon \lambda\in [0,1]\}$ are metric segments. In order to show uniqueness of the metric segments in $\mathcal{F}$, note that we only have to consider the case of two singletons $\{c\}$ where $c\in\stern(C)$, since for every other set $A$ the pair $(A,\{c\})$ appears only once in the definition of $\mathcal{F}$.
  Uniqueness of segments of the form $[\{c_1\},\{c_2\}]$, where $c_1,c_2\in\stern(C)$ follows from the fact that $X$ is hyperbolic.

  For $a,b\in A$, the inequality
  \[
  \rho((1-\lambda) a \oplus \lambda c, (1-\lambda) b \oplus \lambda c)\leq (1-\lambda) \rho(a,b) \leq (1-\lambda) \operatorname{diam}(A),
  \]
  which follows from the fact that $X$ is a hyperbolic space, implies that $(1-\lambda) A\oplus \lambda\{c\}$ is a bounded set. Since it is, by definition, also non-empty and closed, we get that it is contained in $\mathcal{B}(C)$. In addition, note that for all $\mu\in [0,1]$ and all $a\in A$, the point $(1-\mu)a\oplus\mu c$ lies on the metric segment $[a,c]$, which is contained in $C$ because $C$ is $\rho$-star-shaped with respect to~$c$. Therefore $(1-\mu) A \oplus \mu \{c\} \subseteq C$ for all $\mu\in[0,1]$.

  Note that from $h(B,\overline{B})=0$ for arbitrary bounded sets $B\subseteq C$, we may deduce
  \[
  h((1-\lambda) A\oplus \lambda\{c\},E) = h(A^{(1-\lambda)},E)
  \] 
  for every bounded set $E\subseteq C$. Now let $A\in\mathcal{B}(C)$, $c\in\stern(C)$, $\lambda,\mu\in [0,1]$ and assume without loss of generality that $\lambda > \mu$. Then
  \begin{align*}
    h\big((1-\lambda) A\oplus \lambda\{c\},\{c\}\big) & = \sup\{\rho((1-\lambda) a \oplus\lambda c,c)\colon a\in A\} \\
                                                                 & = (1-\lambda) \sup\{\rho(c,a)\colon a\in A\} = (1-\lambda) h(A,\{c\}).
  \end{align*}
  Moreover, we have
  \[
    h\big((1-\mu) A\oplus \mu\{c\}, \{c\}\big) \leq h\big((1-\mu) A\oplus \mu\{c\}, (1-\lambda) A\oplus \lambda\{c\}\big) 
    + h\big((1-\lambda) A\oplus\lambda\{c\}, \{c\}\big),
  \]
  which is equivalent to
  \[
    h\big((1-\lambda) A\oplus \lambda\{c\}, (1-\mu) A\oplus \mu\{c\}\big) \geq (\lambda - \mu) h(A, \{c\}).
  \]
  On the other hand, we also have
  \begin{align*}
    \dist\big((1-\lambda)a\oplus\lambda c, (1-\mu) A\oplus \mu \{c\}\big) 
    & = \inf\{\rho((1-\lambda)a\oplus\lambda c,(1-\mu)b\oplus\mu c)\colon b\in A\}\\
    & \leq (\lambda -\mu) \rho(a,c) \leq (\lambda -\mu) h(A, \{c\})
  \end{align*}
  and analogously, $\dist\big((1-\mu)a\oplus\mu c, (1-\lambda) A \oplus \lambda \{c\}\big) \leq (\lambda-\mu) h(A,\{c\})$.\\
  Therefore $h\big((1-\lambda) A\oplus \lambda\{c\}, (1-\mu) A\oplus \mu\{c\}\big) = |\lambda-\mu| h(A,\{c\})$. The above facts show that for all $A\in\mathcal{B}(C)$, $A\neq\{c\}$, the mapping
  \[
    [0,h(\{c\},A)] \to \mathcal{B}(C), \quad \lambda \mapsto \left(1-\tfrac{\lambda}{h(A,\{c\})}\right) A \oplus \tfrac{\lambda}{h(A,\{c\})}\{c\}
  \]
  is a metric embedding and therefore $[A,\{c\}]$ is a metric segment in $\mathcal{B}(C)$. 

  We now show that $(\mathcal{B}(C), h,\mathcal{F})$ is of temperate curvature. That this triple satisfies condition~\eqref{localuniqueness} of Definition~\ref{def:weaklyhyperbolic} with $D_{\mathcal{B}(C)}=\infty$ is already clear. It only remains to verify condition~\eqref{thinishtriangles} of Definition~\ref{def:weaklyhyperbolic}. We will prove something stronger. Namely, that metric segments in $\mathcal{F}$ even satisfy the hyperbolic inequality~\eqref{eq:hyperineq}, or equivalently
  \begin{equation}\label{eq:hypineq2}
    h((1-\lambda)A\oplus \lambda E,(1-\lambda)B\oplus \lambda E)\leq (1-\lambda) h(A,B)
  \end{equation}
  for all $A,B,E\in\mathcal{B}(C)$ with $[A,E],[B,E]\in\mathcal{F}$. Note that all segments in $\mathcal{F}$ have a set of the form $\{c\}$, where $c\in\stern(C)$, as one of their endpoints. Therefore we only need to verify~\eqref{eq:hypineq2} for the case $E=\left\{c\right\}$ with $c\in \stern(C)$ and the case $A=\left\{c_{1}\right\}$, $B=\left\{c_{2}\right\}$ with $c_{1},c_{2}\in \stern(C)$.
  
  Given $A, B \in\mathcal{B}(C)$ and $c\in\stern(C)$, let $a\in A$ and $b\in B$. Since $X$ is a hyperbolic space, we have
  \[
    \rho((1-\lambda)b\oplus\lambda c, (1-\lambda)a\oplus\lambda c) \leq (1-\lambda) \rho(a,b)
  \]
  and hence 
  \[
    \operatorname{dist}((1-\lambda)b\oplus\lambda c, (1-\lambda) A \oplus \lambda\{c\}) 
    \leq (1-\lambda) \inf \{\rho(a,b)\colon a\in A\} =(1-\lambda) \operatorname{dist}(b,A)
  \]
  for all elements of $(1-\lambda) B\oplus \lambda\{c\}$. Since the situation is completely analogous if we swap the roles of $(1-\lambda) A \oplus \lambda\{c\}$ and $(1-\lambda) B \oplus \lambda \{c\}$, we may conclude that
  \[
    h((1-\lambda) A \oplus \lambda\{c\},(1-\lambda) B \oplus \lambda\{c\})\leq (1-\lambda) h(A,B).
  \]
  This verfies inequality~\eqref{eq:hypineq2} for the case $E=\left\{c\right\}$. 
  
  To prove the inequality in the remaining case, we take $c_1,c_2\in\stern(C)$, $E\in\mathcal{B}(C)$ and observe that
  \[
    \rho((1-\lambda) c_1\oplus \lambda a, (1-\lambda) c_2\oplus \lambda a') \leq (1-\lambda) \rho(c_1,c_2) + \lambda \rho(a,a'),
  \]
  for all $a,a'\in E$, by~\eqref{eq:hyperineq}. From this we may deduce
  \[
    \operatorname{dist}((1-\lambda) c_1\oplus \lambda a, (1-\lambda)\{c_2\}\oplus\lambda E) \leq (1-\lambda) \rho(c_1,c_2) = (1-\lambda) h(\{c_1\},\{c_2\})
  \]
  for all $a\in E$, and therefore, since the situation is completely symmetric with respect to $c_1$ and $c_2$,
  \[
    h((1-\lambda)\{c_1\}\oplus\lambda E, (1-\lambda) \{c_2\}\oplus\lambda E) \leq (1-\lambda) h(\{c_1\},\{c_2\}).
  \]
  Finally, note that by the construction of $\mathcal{F}$, we get 
  \[
    \stern(\mathcal{B}(C)) = \{\{c\}\colon c\in\stern(C)\} 
  \]
  and hence $\mathcal{B}(C)$ is a $h$-star-shaped subset of $(\mathcal{B}(C),h,\mathcal{F}$).
\end{proof}

\begin{remark}
  Note that the above construction does not work if we replace the set $\{c\}$ by a non-singleton as can be seen by the following example. We consider the metric space $C:=[-1,1]^2$ equipped with the standard metric and set $A:=\{(-1,-1),(-1,1)\}$ and $B:=\{(1,-1),(1,1)\}$. We get $h(A,B)=2$ and
  \[
    \frac{1}{2}A+\frac{1}{2} B = \{(0,-1),(0,0),(0,1)\}.
  \]
  Therefore $h(\frac{1}{2}A+\frac{1}{2} B, A)= \sqrt{2} \neq \frac{1}{2} h(A,B)$. More generally, Example~4.7 in~\cite{Str2014Porosity} shows that even in the case of Banach spaces the hyperspace of bounded and closed subsets cannot be a hyperbolic space in the sense of Reich-Shafrir.
\end{remark}

As a consequence of Lemma~\ref{lem:starShapedSetValued} and Theorems~\ref{thm:Nsigporous}, \ref{thm:convexresidual} and~\ref{thm:starshapedresidual}, we can infer the following corollary regarding set-valued 
nonexpansive mappings.

\begin{corollary}\label{cor:OnHyperspaces}
  Let $X$ be a complete hyperbolic space and $C\subseteq X$ be a non-empty, non-singleton, closed, $\rho$-star-shaped subset. Then the following statements hold:
  \begin{enumerate}[(i)]
  \item The set 
    \[
      \mathcal{N}(C,\mathcal{B}(C)) := \{f\colon C\to \mathcal{B}(C)\colon \lip(f)<1\},
    \]
    is a $\sigma$-porous subsets of the space
    \[
      \mathcal{M}(C,\mathcal{B}(C)) := \{f\colon C\to \mathcal{B}(C)\colon \lip(f)\leq 1\}
    \]
    of all nonexpansive $\mathcal{B}(C)$-valued mappings equipped with the metric $d_\theta$.
  \item If $C$ is separable, there exists a $\sigma$-porous set $\widetilde{\N}\subseteq \M(C,\mathcal{B}(C))$ such that for all $f\in \M(C,\mathcal{B}(C))\setminus\widetilde{\N}$, the set
    \begin{equation*}
      \widehat{R}(f)=\left\{x\in C\colon \globlip(f,x)=1\right\}
    \end{equation*}
    is a residual subset of $C$.
  \item If $C$ is separable and $\rho$-convex, there exists a $\sigma$-porous set $\widetilde{\N}\subseteq \M(C,\mathcal{B}(C))$ such that for all $f\in \M(C,\mathcal{B}(C))\setminus\widetilde{\N}$, the set
    \begin{equation*}
      R(f)=\left\{x\in C\colon \lip(f,x)=1\right\}
    \end{equation*}
    is a residual subset of $C$.
  \end{enumerate}
\end{corollary}

\begin{remark}
  Results analogous to Corollary~\ref{cor:OnHyperspaces} are valid for all hyperspaces $\mathcal{X}(C)$ with the property that
  \[
    (1-\lambda) A\oplus \lambda\{c\} \in \mathcal{X}(C),
  \]
  where $(1-\lambda) A\oplus \lambda\{c\}$ is defined in~\eqref{eq:SetConvex}, for all $c\in\stern(C)$, $\lambda\in[0,1]$ and $A\in\mathcal{X}(C)$. In the case of $\mathcal{K}(C)$ this follows 
  from the fact that for all $c\in\stern(C)$ and all $\lambda\in[0,1]$, the mapping
  \[
    C\to C,\quad a\mapsto (1-\lambda) a\oplus \lambda c
  \]
  is continuous.
  In~\cite{PL2014ContractiveSetValued} spaces with this property are called ``admissible'' and, besides $\mathcal{B}(C)$ and $\mathcal{K}(C)$, the following examples are given in~\cite[Remark~2.5, p.~1417]{PL2014ContractiveSetValued}: the space of singletons, the space of bounded, closed and $\rho$-convex sets, and the space of compact and $\rho$-convex sets.
\end{remark}

\begin{remark}
  In addition to the above corollary, we can also show that the set of bounded strict contractions is a $\sigma$-porous subset of the space of all bounded nonexpansive $\mathcal{B}(C)$- and $\mathcal{K}(C)$-valued mappings if we equip these spaces with the metric of uniform convergence.
\end{remark}

\begin{remark}
  Note that if $X$ is a Banach space, we do not need to take the closure in the definition of the set $(1-\lambda)A\oplus \lambda\{c\}$ in~\eqref{eq:SetConvex} since the sum of a closed set and a compact set is closed. In addition, if we define
  \begin{equation}\label{eq:ConvComb}
    (1-\lambda)A\oplus \lambda B := \overline{\{(1-\lambda)a + \lambda b\colon a\in A,\; b\in B\}}.
  \end{equation}
  for bounded, closed and convex sets $A$ and $B$ and $\lambda\in[0,1]$ we get analogously to above a well-defined mapping from $[0,h(A,B)]$ to the space of bounded, closed and convex sets which satisfies the hyperbolicity inequality. That the above mapping is an isometry follows from this inequality and from
  \[
    (1-\lambda)A\oplus\lambda B = \frac{1-\lambda}{1-\mu}\big((1-\mu)A\oplus\mu B\big)\oplus \frac{\lambda-\mu}{1-\mu}B.
  \]
  for bounded, closed and convex sets $A$ and $B$ and $0\leq \mu < \lambda\leq 1$, which can be shown by interchanging the occurring convex combinations.
  This implies that the space of bounded, closed and convex subsets of a closed and convex subset of a Banach space is $h$-convex. In particular, the hyperspace of bounded, closed and convex subsets of a bounded and closed subset of a Banach space is a hyperbolic space. We remark in passing that convexity, in a more general sense, of hyperspaces of compact sets is studied in detail in~\cite{Dud1970ConvexV}. For the star-shapedness and hyperbolicity properties of these hyperspaces on subsets of Banach spaces, we refer the interested reader to~\cite{Str2014Porosity}.
\end{remark}

\vspace{3mm}
{\noindent\textbf{Acknowledgments.}} The authors wish to thank an anonymous referee for reading the paper very carefully and for many useful and interesting suggestions which made the article significantly more reader friendly. This research was supported in part by the Israel Science Foundation (Grant 389/12), the Fund for the Promotion of Research at the Technion and by the Technion General Research Fund.

\newpage
\noindent
Christian Bargetz\\
Department of Mathematics\\
The Technion---Israel Institute of Technology\\
32000 Haifa,
Israel\\[1mm]
and\\[1mm]
Department of Mathematics\\
University of Innsbruck\\
Technikerstraße 13,
6020 Innsbruck,
Austria (current address)\\
\texttt{christian.bargetz@uibk.ac.at}\\[3mm]
Michael Dymond\\
Department of Mathematics\\
University of Innsbruck\\
Technikerstraße 13,
6020 Innsbruck,
Austria\\
\texttt{michael.dymond@uibk.ac.at}\\[3mm]\noindent
Simeon Reich\\
Department of Mathematics\\
The Technion---Israel Institute of Technology\\
32000 Haifa,
Israel\\
\texttt{sreich@math.technion.ac.il}


\begin{thebibliography}{10}

\bibitem{Ban1922EnsembesAbstraits}
S.~Banach.
\newblock {Sur les op\'erations dans les ensembles abstraits et leur
  application aux \'equations int\'egrales.}
\newblock \emph{Fund. Math.}, 3 (1922):133--181.

\bibitem{BD2016:Porosity}
C.~Bargetz and M.~Dymond.
\newblock $\sigma$-{P}orosity of the set of strict contractions in a space of
  non-expansive mappings.
\newblock \emph{Israel J. Math.}, 214 (2016):235--244.

\bibitem{BL2000GeomtericNonlinear}
Y.~Benyamini and J.~Lindenstrauss.
\newblock \emph{Geometric {N}onlinear {F}unctional {A}nalysis.}
\newblock American Mathematical Society, Providence, RI (2000).

\bibitem{BH1999MetricSpaces}
M.~R. Bridson and A.~Haefliger.
\newblock \emph{{M}etric {S}paces of {N}on-positive {C}urvature}, volume 319 of
  \emph{Grundlehren der Mathematischen Wissenschaften}.
\newblock Springer-Verlag, Berlin (1999).

\bibitem{Bro1911Abbildungen}
L.~E.~J. Brouwer.
\newblock \"{U}ber {A}bbildung von {M}annigfaltigkeiten.
\newblock \emph{Math. Ann.}, 71 (1911):97--115.

\bibitem{Bro1965HilbertSpace}
F.~E. Browder.
\newblock Fixed-point theorems for noncompact mappings in {H}ilbert space.
\newblock \emph{Proc. Nat. Acad. Sci. U.S.A.}, 53 (1965):1272--1276.

\bibitem{Bul1984DenjoyIndex}
P.~S. Bullen.
\newblock Denjoy's index and porosity.
\newblock \emph{Real Anal. Exchange}, 10 (1984/85):85--144.

\bibitem{Bus1948NonpositiveCurvature}
H.~Busemann.
\newblock Spaces with non-positive curvature.
\newblock \emph{Acta Math.}, 80 (1948):259--310.

\bibitem{DM1976Convergence}
F.~S. de~Blasi and J.~Myjak.
\newblock Sur la convergence des approximations successives pour les
  contractions non lin\'eaires dans un espace de {B}anach.
\newblock \emph{C. R. Acad. Sci. Paris S\'er. A-B}, 283 (1976):A185--A187.

\bibitem{DM1989Porosity}
F.~S. de~Blasi and J.~Myjak.
\newblock Sur la porosit\'e de l'ensemble des contractions sans point fixe.
\newblock \emph{C. R. Acad. Sci. Paris S\'er. I Math.}, 308 (1989):51--54.

\bibitem{BMRZ2009GenericExistence}
F.~S. de~Blasi, J.~Myjak, S.~Reich, and A.~J. Zaslavski.
\newblock Generic existence and approximation of fixed points for nonexpansive
  set-valued maps.
\newblock \emph{Set-Valued Var. Anal.}, 17 (2009):97--112.

\bibitem{Den1941LeconsII}
A.~Denjoy.
\newblock \emph{Le\c cons sur le {C}alcul des {C}oefficients d'une {S}\'erie
  {T}rigonom\'etrique. {T}ome {II}. {M}\'etrique et {T}opologie d'{E}nsembles
  {P}arfaits et de {F}onctions}.
\newblock Gauthier-Villars, Paris (1941).

\bibitem{dolvzenko1967granivcnye}
E.~P. Dol{\v{z}}enko.
\newblock Boundary properties of arbitrary functions.
\newblock \emph{Izv. Akad. Nauk SSSR Ser. Mat.}, 31 (1967):3--14.
\newblock (in Russian).

\bibitem{Dud1970ConvexV}
R.~Duda.
\newblock On convex metric spaces. {V}.
\newblock \emph{Fund. Math.}, 68 (1970):87--106.

\bibitem{Esp1015ContinuousSelections}
R.~Esp{\'{\i}}nola and A.~Nicolae.
\newblock Continuous selections of {L}ipschitz extensions in metric spaces.
\newblock \emph{Rev. Mat. Complut.}, 28 (2015):741--759.

\bibitem{GK1990FixedPointTheory}
K.~Goebel and W.~A. Kirk.
\newblock \emph{{T}opics in {M}etric {F}ixed {P}oint {T}heory}, volume~28 of
  \emph{Cambridge Studies in Advanced Mathematics}.
\newblock Cambridge University Press, Cambridge (1990).

\bibitem{GR1984UniformConvexity}
K.~Goebel and S.~Reich.
\newblock \emph{{U}niform {C}onvexity, {H}yperbolic {G}eometry, and
  {N}onexpansive {M}appings}, volume~83 of \emph{Monographs and Textbooks in
  Pure and Applied Mathematics}.
\newblock Marcel Dekker, Inc., New York and Basel (1984).

\bibitem{Jen1994ModernGeometry}
G.~A. Jennings.
\newblock \emph{{M}odern {G}eometry with {A}pplications}.
\newblock Universitext. Springer-Verlag, New York (1994).

\bibitem{Koh2005Metatheorems}
U.~Kohlenbach.
\newblock Some logical metatheorems with applications in functional analysis.
\newblock \emph{Trans. Amer. Math. Soc.}, 357 (2005):89--128.

\bibitem{Kur1966Topology}
K.~Kuratowski.
\newblock \emph{Topology. {V}ol. {I}}.
\newblock New edition, revised and augmented. Translated from the French by J.
  Jaworowski. Academic Press, New York-London; Pa\'nstwowe Wydawnictwo Naukowe,
  Warsaw (1966).

\bibitem{PL2014ContractiveSetValued}
L.-H. Peng and X.-F. Luo.
\newblock Contractive set-valued maps in hyperbolic spaces.
\newblock \emph{J. Nonlinear Convex Anal.}, 16 (2015):1415--1424.

\bibitem{Pia2011HalpernIteration}
B.~Pi{\c{a}}tek.
\newblock Halpern iteration in {${\mathrm CAT}(\kappa)$} spaces.
\newblock \emph{Acta Math. Sin. (Engl. Ser.)}, 27 (2011):635--646.

\bibitem{Pia2015FixedPointProperty}
B.~Pi\c{a}tek.
\newblock The fixed point property and unbounded sets in spaces of negative
  curvature.
\newblock \emph{Israel J. Math.}, 209 (2015):323--334.

\bibitem{Rak1962Contractive}
E.~Rakotch.
\newblock A note on contractive mappings.
\newblock \emph{Proc. Amer. Math. Soc.}, 13 (1962):459--465.

\bibitem{Rei2005GenericityPorosity}
S.~Reich.
\newblock Genericity and porosity in nonlinear analysis and optimization.
\newblock In \emph{ESI Preprint 1756. Proceedings of CMS'05 (Computer Methods
  and Systems), Krak\'ow 2005}, pages 9--15 (2005).

\bibitem{RS1990Nonexpansive}
S.~Reich and I.~Shafrir.
\newblock Nonexpansive iterations in hyperbolic spaces.
\newblock \emph{Nonlinear Anal.}, 15 (1990):537--558.

\bibitem{RZ2001NoncontractiveMappings}
S.~Reich and A.~J. Zaslavski.
\newblock The set of noncontractive mappings is {$\sigma$}-porous in the space
  of all nonexpansive mappings.
\newblock \emph{C. R. Acad. Sci. Paris S\'er. I Math.}, 333 (2001):539--544.

\bibitem{RZ2016TwoPorosity}
S.~Reich and A.~J. Zaslavski.
\newblock Two porosity theorems for nonexpansive mappings in hyperbolic spaces.
\newblock \emph{J. Math. Anal. Appl.}, 433 (2016):1220--1229.

\bibitem{Ren1995Porosity}
D.~L. Renfro.
\newblock Porosity, nowhere dense sets and a theorem of {D}enjoy.
\newblock \emph{Real Anal. Exchange}, 21 (1995/96):572--581.

\bibitem{Str2012PorousAndMeager}
F.~Strobin.
\newblock Some porous and meager sets of continuous mappings.
\newblock \emph{J. Nonlinear Convex Anal.}, 13 (2012):351--361.

\bibitem{Str2014Porosity}
F.~Strobin.
\newblock {$\sigma$}-porous sets of generalized nonexpansive mappings.
\newblock \emph{Fixed Point Theory}, 15 (2014):217--232.

\bibitem{Zaj2005Porous}
L.~Zaj{\'{\i}}{\v{c}}ek.
\newblock On {$\sigma$}-porous sets in abstract spaces.
\newblock \emph{Abstr. Appl. Anal.}, 2005 (2005):509--534.

\end{thebibliography}
\end{document}